\documentclass[11pt]{article}
 \usepackage[utf8]{inputenc}
 \usepackage{amsmath}
 \usepackage{amsfonts}
 \usepackage{amssymb}
 \usepackage{graphicx}
 \usepackage{amsmath}
 \usepackage{amsthm}
 \usepackage{enumerate}
 \usepackage{verbatim}
 \usepackage{hyperref}
 \usepackage{lmodern,eurosym,ae,tabularx,amssymb,url,epsf,epsfig,amsmath,amsfonts,graphicx,dsfont}
 \usepackage{amsmath, amsfonts,amssymb}
 \usepackage{amsthm}
 \usepackage{bbm}
 \usepackage{color}
 \usepackage{relsize}
 \usepackage{textcomp}
 \usepackage{graphicx,color,bm}
 \usepackage{cancel}
 
 \newlength{\defbaselineskip}
 \usepackage[a4paper,top=2cm,bottom=2cm,left=2cm,right=2cm]{geometry}
 
 \usepackage{graphics}
 \usepackage{pgfplots}
 \pgfplotsset{width=7cm,compat=newest} 
 \usepackage{tikz}
 \usetikzlibrary{shapes.geometric}
 \usepackage{tikz-cd}
 \usetikzlibrary{patterns}
 \usepackage{caption}
 \usepackage{subcaption}
 
 \usepackage{float}
 
 \usepackage{pgfplotstable}
 \usetikzlibrary{intersections}
 \usetikzlibrary{automata,positioning,arrows}
 
\DeclareMathOperator{\supp}{supp}

\def \N{\mathbb{N}}
\def \R{\mathbb{R}}

\def \E{\mathbb{E}}
\def \P{\mathbb{P}}
\def \L{\mathcal{L}}

\def \m{\mathfrak{m}}

\def \O{\mathcal{O}}
\def \Ex{\mathcal{E}}

\theoremstyle{plain} 
\newtheorem{theorem}{Theorem}[section]

\newtheorem{lemma}[theorem]{Lemma}

\newtheorem{proposition}[theorem]{Proposition}
\newtheorem{remark}[theorem]{Remark}

\date{}

\begin{document}
\title{\bf Properties of the American price function in the Heston-type models}

\author{{\sc Damien Lamberton}\thanks{%
		Université Paris-Est, Laboratoire d'Analyse et de Mathématiques Appliquées (UMR 8050), UPEM, UPEC, CNRS, Projet Mathrisk INRIA, F-77454, Marne-la-Vallée, France - {\tt damien.lamberton@u-pem.fr}}\\
	{\sc Giulia Terenzi}\thanks{%
		Université Paris-Est, Laboratoire d'Analyse et de Mathématiques Appliquées (UMR 8050), UPEM, UPEC, CNRS, Projet Mathrisk INRIA, F-77454, Marne-la-Vallée, France, and Universit\`a di Roma Tor Vergata, Dipartimento di Matematica, Italy - {\tt terenzi@mat.uniroma2.it}}\\}
%
\maketitle
\begin{abstract}\noindent{\parindent0pt}
	We study some properties of the American option price in the  stochastic volatility Heston model. We first prove that, if the payoff function is convex and satisfies some regularity assumptions, then the option value function is increasing with respect to the  volatility variable. Then, we focus on the standard put option and we extend to the Heston model some well known results in the Black and Scholes world, most by using probabilistic techniques. In particular, we study  the  exercise boundary, we prove the strict convexity of the value function in the continuation region, we extend to this model the early exercise premium formula and we prove a weak form of the smooth fit property.
\end{abstract}


\smallskip
\section{Introduction}
The Black and Scholes model (1973) was the starting point of equity dynamics modelling and it is still  widely used as a useful approximation. Nevertheless, it is a well known fact that it disagrees with reality in a number of significant ways and even one of the authors, F. Black, in 1988 wrote about the flaws of the model. Indeed, empirical studies show that in the real market the log-return process is not normally distributed and its distribution is often affected by heavy tail and high peaks. Moreover,  the 
assumption of a constant  volatility turns out to be too rigid to model the real world financial market. 

These limitations have called for more sophisticated models which can better reflect the reality and the fact that volatility should vary randomly is now completely recognized. 
A large body of literature was devoted to the so called stochastic volatility models, where the volatility  is modelled by an autonomous stochastic process driven by some additional random noise.  In this context, the celebrated model introduced by S. Heston in 1993 \cite{H} is one of the most widely used stochastic volatility models in the financial world and it was the starting point for several generalizations.

One of the strengths of the Black and Scholes type models relies in their analytical tractability. 
A large number of papers have been devoted to the pricing of European and American options and to the study of the regularity properties of the price in this framework.  

Things become more complicated in the case of stochastic volatility models. Some properties of European options were studied, for example, in \cite{Oa} but if we consider American options, as far as we know, the existing literature is rather poor. One of the main reference  is a paper by Touzi \cite{T}, in which  the author studies some  properties of a standard American put option in a class of stochastic volatility models under  classical assumptions, such as the  uniform ellipticity of the model. 

However, the assumptions in \cite{T}  are not satisfied by the  Heston  model because of its degenerate nature. In fact, the infinitesimal generator associated with the two dimensional diffusion given by the log-price process and the volatility process is not uniformly elliptic: it degenerates on the boundary of the domain, that is when the volatility variable vanishes. Therefore, the analytical characterization of an American option value does not follow from the classical theory of parabolic obstacle problems (we study this topic in  details in \cite{LT}) and some of the analytical techniques used in \cite{T} cannot be directly applied. 

This paper is devoted to the study of some properties of the American option price in the Heston model. Our main aim is to extend some well known results in the Black and Scholes world to the Heston type stochastic volatility models. We do it  mostly by using probabilistic techniques.

In more details, the paper is organized as follows. In Section 2 we recall the model and we set up our notation. In Section 3, we prove that, if the payoff function is convex and satisfies some regularity assumptions, the American option value function is increasing with respect to the volatility variable. This topic was already addressed in \cite{AOJ} with an elegant probabilistic approach, under the assumption that the coefficients of the model satisfy the well known Feller condition. Here, we prove it without imposing  conditions on the coefficients.

 Then, in Section 4 we focus on the standard American put option. We first generalise to the Heston model the well known notion of  critical price or exercise boundary  and we study some properties of this boundary. Then we prove that the American option price is strictly convex in the continuation region. This result was already proved in \cite{T} for uniformly elliptic stochastic volatility  by using PDE techniques.  Here, we extend the result to the degenerate Heston model by using a probabilistic approach.  We also give an explicit formulation of the early exercise premium, that  is the difference in price between an American option and an otherwise identical European option, and we do it by using  results first  introduced in \cite{J}. Finally,  we  provide  a weak form of the so called smooth fit property. 
The paper ends with an appendix, which is devoted to the proofs of some technical results.
	\section{The American option price in the Heston model}
We recall that in the stochastic volatility Heston model  the asset price $S$ and the volatility process $Y$ evolve under the pricing measure according to  the stochastic differential equation system
	\begin{equation}\label{hest}
		\begin{cases}
			\frac{dS_t}{S_t}=(r-\delta)dt + \sqrt{Y_t}dB_t,\qquad&S_0=s>0,\\ dY_t=\kappa(\theta-Y_t)dt+\sigma\sqrt{Y_t}dW_t, &Y_0=y\geq 0,
		\end{cases}
	\end{equation}
	where $B$ and $W$ denote two correlated Brownian motions with 
	$$
	d\langle B,W \rangle_t=\rho dt, \qquad \rho \in (-1,1).
	$$
	Here $r>0$ and $\delta\geq0$ are respectively  the risk free rate of interest and the continuous dividend rate. The dynamics of $Y$ follows a CIR process with mean reversion rate $\kappa>0$ and long run state $\theta> 0$. The parameter $\sigma>0$ is called the volatility of the volatility. It is well known that under the so called Feller condition on the coefficients, that is if  $2\kappa\theta \geq \sigma^2$, the process $Y$ with starting condition $Y_0=y>0$ remains always positive. On the other hand, if the Feller condition is not satisfied, $Y$ reaches zero with probability one for any $Y_0=y\geq 0$ (see, for example, \cite{Abook}. Otherwise stated, in this paper we do not assume that the Feller condition holds: in general, the process $Y$ can vanish.	
	
	We denote by $\L$ the infinitesimal generator of the pair $(S,Y)$, that is the differential operator given by
	\begin{equation}\label{L}
	\L= \frac y 2 \left(s^2 \frac{\partial^2}{\partial s^2 } +2s\rho \sigma  \frac{\partial^2}{\partial s \partial y }+ \sigma^2  \frac{\partial^2}{\partial y^2 }\right)+ \left(r-\delta \right)s\frac{\partial}{\partial s} +\kappa(\theta - y)\frac{\partial}{\partial y}.
	\end{equation}
Let  $(S^{t,s,y}_u,Y^{t,y}_u)_{u\in [t,T]}$ be  the solution of \eqref{hest} which starts  at time $t$ from the position  $(s,y)$. When the initial time is $t=0$ and there is no ambiguity, we will often write $(S^{s,y}_u,Y^{y}_u)$ or directly $(S_u,Y_u)$  instead of $(S^{0,s,y}_u,Y^{0,y}_u)$.
	In this framework,  the price of an American option with a nice enough payoff $(\varphi(S_t))_{t\in [0,T]} $ and maturity $T$ is given by $P_t=P(t,S_t,Y_t)$, where 

\begin{equation*}
P(t,s,y)=\sup_{\tau \in\mathcal T_{t,T}}\E[e^{-r(\tau-t )}\varphi(S^{t,s,y}_\tau)],
\end{equation*}
 $\mathcal{T}_{t,T}$ being  the set of the stopping times with values in $[t,T]$.
 
It  will be useful to consider  the log-price process  $X_t=\log S_t$. In this case, recall that the pair $(X,Y)$  evolves according to
 \begin{equation}
 \begin{cases}\label{hest2}
 dX_t=\left(r-\delta-\frac{Y_t}{2}\right)dt + \sqrt{Y_t}dB_t, \qquad&X_0=x=\log s\in\R,\\ dY_t=\kappa(\theta-Y_t)dt+\sigma\sqrt{Y_t}dW_t, &Y_0=y\geq 0,
 \end{cases}
 \end{equation}
 and has infinitesimal generator given by
 \begin{equation}\label{Llog}
 \tilde {\mathcal{L}}= \frac y 2 \left( \frac{\partial^2}{\partial x^2 } +2\rho \sigma  \frac{\partial^2}{\partial x \partial y }+ \sigma^2  \frac{\partial^2}{\partial y^2 }\right)+ \left(r-\delta-\frac y 2 \right)\frac{\partial}{\partial x} +\kappa(\theta - y)\frac{\partial}{\partial y}.
 \end{equation}
Note that $\tilde \L$ has unbounded coefficients and it is not uniformly elliptic: it degenerates on the boundary of the definition set $\O=\R\times (0,\infty)$, that is when $y=0$.

 With this change of variables, the American option price function is given by $u(t,x,y)=P(t,e^x,y)$, which can be rewritten as
 \begin{equation*}
 u(t,x,y)=\sup_{\tau \in \mathcal{T }_{t,T}}\E[e^{-r(\tau-t)}\psi(X^{t,x,y}_\tau)],
 \end{equation*}
 where $\psi(x)=\varphi(e^x)$.

	\section{Monotonicity with respect to the volatility}\label{sect-monotony}
In this section we  prove the increasing feature of the option price with respect to the volatility variable under the assumption that the payoff function $\varphi$ is convex and satisfies some regularity properties. The same topic was addressed by Touzi in \cite{T}  for uniformly elliptic stochastic volatility models  and by Assing \textit{et al.} \cite{AOJ} for a class of models which includes the Heston model when the Feller condition is satisfied.

  For convenience we pass to the logarithm in the $s-$variable and we study the monotonicity of the function $u$. Note that the convexity assumption on the payoff function $\varphi\in  C^2(\R)$ corresponds to the condition $ \psi''-\psi'\geq 0$ for the function $\psi(x)=\varphi(e^x)$.

 Let us recall some standard notation. For $\gamma >0$ we introduce the following weighted Sobolev spaces 
$$
L^2(\R,e^{-\gamma|x|})=\left \{ u:\R\rightarrow\R:\|u\|_2^2=\int u^2(x)e^{-\gamma|x|}dx<\infty \right   \},
$$
$$
W^{1,2}(\R,e^{-\gamma|x|})=\left \{ u\in L^2(\R,e^{-\gamma|x|}): \frac{\partial u }{\partial x}\in L^2(\R,e^{-\gamma|x|}) \right   \},
$$
$$
W^{2,2}(\R,e^{-\gamma|x|})= \left  \{ u\in L^2(\R,e^{-\gamma|x|}): \frac{\partial u }{\partial x},\frac{\partial^2 u }{\partial x^2}\in L^2(\R,e^{-\gamma|x|})   \right  \}.
$$
\begin{theorem}\label{monotonie}
	Let  $\psi$ be a bounded function such that $\psi\in W^{2,2}(\R,e^{-\gamma |x|})\cap  C^2(\R)$ and   $\psi''-\psi'\geq 0$.
Then the value function $u$ is nondecreasing with respect to the volatility variable.
\end{theorem}
In order to prove Theorem \ref{monotonie}, let us consider   a smooth approximation $f_n\in C^\infty(\R)$ of the function $f(y)=\sqrt{y^+}$, such that $f_n$ has  bounded derivatives, $1/n\leq f_n\leq n$, $f_n(y)$ is increasing in $y$, $f_n^2$ is Lipschitz continuous uniformly in $n$ and $f_n\rightarrow f$ locally uniformly  as $n\rightarrow \infty$. 

Then, we consider  the sequence of SDEs
 \begin{equation}
 \begin{cases}\label{hestapprox}
 dX^n_t=\left(r-\delta-\frac{f_n^2(Y^n_t)} 2 \right) dt + f_n(Y^n_t)dB_t,\qquad &X^n_0=x,\\ dY^n_t=\kappa\left(\theta-f_n^2(Y^n_t)\right)dt+\sigma f_n(Y^n_t)dW_t, &Y_0^n=y.
 \end{cases}
 \end{equation}
 Note that, for every $n\in\N$, the diffusion matrix $a_n(y)=\frac 1 2 \Sigma_n(y) \Sigma_n(y)^t$, where 
 		$$
 		\Sigma_n(y) = \left(
 		\begin{array}{cc}
 	\sqrt{1-\rho^2}	f_n(y) &  \rho f_n(y) \\ 
 		0 &  \sigma f_n(y)
 		\end{array} 	\right),
 		$$
 		 is uniformly elliptic.
 For any fixed $n\in\N$ the infinitesimal generator of the diffusion $(X^n,Y^n)$ is given by
 $$
\tilde{ \mathcal{L}}^n= \frac {f_n^2(y)} 2\left( \frac{\partial^2}{\partial x^2} +2\rho\sigma \frac{\partial^2}{\partial x \partial y}+\sigma^2 \frac{\partial^2}{\partial y^2}\right) +\left(r-\delta-\frac{f_n^2(y)} 2 \right) \frac{\partial }{\partial x} + \kappa\left(\theta-f_n^2(y)\right)\frac{\partial }{\partial y}
 $$
 and it is uniformly elliptic with  bounded coefficients.
 
 We will need the following result.
 \begin{lemma}\label{lemmasup_diff}
 	For any $\lambda>0$, we have\begin{equation}
 	\label{1} \lim_{n \rightarrow \infty} \P\left( \sup_{t \in [0,T]}|X^n_t-X_t|\geq \lambda \right)=0 \end{equation}  and\begin{equation}
 	\label{2}
 	\lim_{n \rightarrow \infty} \P\left(\sup_{t \in [0,T]}|Y^n_t-Y_t|\geq \lambda \right)=0 .\end{equation}  
 \end{lemma}
 The proof is inspired by the proof of uniqueness of the solution for the CIR process (see \cite[Section IV.3]{IW}). We postpone it to the Appendix.
 
From now on, let us set $\E_{x,y}[\cdot]=\E[\cdot|(X_0,Y_0)=(x,y)]$.  For every $n \in \N$,	we consider the American value function with payoff $\psi$ and underlying diffusion $(X^n,Y^n)$,  that is
 $$
 u^n(t,x,
 y)=   \sup_{\tau \in \mathcal{T}_{0,T-t}} \E_{x,y} \left[   e^{-r\tau} \psi(X_\tau^{n})       \right],\qquad  (t,x,y) \in [0,T]\times \R \times [0,\infty).
 $$
 We prove that $u^n$ is actually an approximation of the function $u$, at least for bounded continuous payoff functions.
 \begin{proposition}\label{convP}
 Let $\psi$ be a bounded continuous function.
 	Then, 
 	$$
 	\lim_{n \rightarrow \infty }|	u^n(t,x,y) -u(t,x,y) |=0,\qquad  (t,x,y) \in [0,T]\times \R \times [0,\infty).
 	$$ 
 \end{proposition}
 \begin{proof}
 	For any $\lambda>0$,
 	\begin{align*}
 	\bigg|      \sup_{\tau \in \mathcal{T}_{0,T-t}} \E_{x,y} &\left[   e^{-r\tau} \psi(X_\tau^{n})       \right]-  \sup_{\tau \in \mathcal{T}_{0,T-t}} \E_{x,y} \left[   e^{-r\tau}  \psi(X_\tau)      \right] \bigg|\\
 	&\leq  \sup_{\tau \in \mathcal{T}_{0,T-t}}   	\bigg|  \E_{x,y} \left[   e^{-r\tau}  (\psi(X_\tau^{n})    -\psi(X_\tau)  ) \right] \bigg|\\
 	&\leq \E_{x,y} \left[ \sup_{t\in[0,T]}  | \psi(X_t^{n})    -\psi(X_t)   |   \right]\\
 	& \leq   \E_{x,y} \left[ \sup_{t\in[0,T]}  | \psi(X_t^{n})    -\psi(X_t)    |  \mathbf{1}_{\{|X^n_t-X_t|\leq \lambda\}  }\right]+ 2 \|\psi\|_\infty \P\left(  \sup_{t\in[0,T]}  |X^n_t-X_t |>\lambda \right).
 	\end{align*}
 	Then the assertion easily follows using \eqref{1} and the arbitrariness of $\lambda$.
 \end{proof}

We can now prove that, for every $n\in\N$,  the approximated price function $u^n$ is nondecreasing with respect to the volatility variable. 
\begin{proposition}\label{mon2}
Assume that $\psi \in W^{2,2}(\R,e^{-\gamma|x|}dx)\cap C^2(\R)$ and $\psi''-\psi'\geq 0$. Then  $\frac{\partial u^n}{\partial y }\geq 0$ for every $n\in\N$.
\end{proposition}
\begin{proof}
Fix $n\in\N$.  We know from the classical theory of variational inequalities that
$u^n$ is the unique solution of the associated variational inequality (see, for example, \cite{JLL}). Moreover, 
$u^n$ is the limit of the solutions of a sequence of  penalized problems. In particular, consider a family of penalty functions $\zeta_\varepsilon:\R\rightarrow\R$  such that, for each $\varepsilon>0$, 
$\zeta_\varepsilon$ is a  $C^2 $, nondecreasing and concave function with bounded derivatives, satisfying $\zeta_\varepsilon(u)=0$, for $u\geq 	\varepsilon$ and $\zeta_\varepsilon(0)=b$, where $b$ is such that $\tilde{\mathcal{A}}^n\psi\geq b$ with the notation $\tilde{\mathcal{A}}^n=\tilde \L^n-r$ (see the proof of Theorem 3 in \cite{damien}). Then,  there exists a sequence  $(u^n_\varepsilon)_{\varepsilon>0}$ such that $\lim_{\varepsilon\rightarrow 0 }u^n_\varepsilon=u^n$ in the sense of distributions and, for every $\varepsilon >0$, 
\begin{equation*}
\begin{cases}
-\frac{\partial u^n_\varepsilon}{\partial t}- \mathcal{A}^n u^n_\varepsilon + \zeta_\varepsilon(u^n_\varepsilon-\psi)=0,\\
u^n_\varepsilon(T)=	\psi(T).
\end{cases}
\end{equation*}
In order to simplify the notation, hereafter in this proof we denote  by $u$ the function $u^n_\varepsilon$. 

Recall that, from the classical theory of parabolic semilinear equations, since $\psi\in C^2(\R)$ we have that $u\in  C^{2,4}([0,T), \R\times (0,\infty))$ (here we refer, for example, to \cite[Chapter VI]{LSU}). 
Set now $\bar u= \frac{\partial u }{\partial y }$. Differentiating the equation satisfied by $u^n$, since $\psi$ does not depend on $y$, we get that $\bar u$ satisfies 
\begin{equation*}
\begin{cases}
-\frac{\partial \bar u}{\partial t}-\bar{ \mathcal{A}}^n \bar u=  f_n(y)f_n'(y) \left(\frac{\partial^2 u}{\partial x^2}-\frac{\partial u}{\partial x}\right),\\
\bar u(T)=0,
\end{cases}
\end{equation*}
where 
\begin{align*}
\bar{\mathcal{A}}^n&= \frac {f_n^2(y)} 2\left( \frac{\partial^2}{\partial x^2} +2\rho\sigma \frac{\partial^2}{\partial x \partial y}+\sigma^2 \frac{\partial^2}{\partial y^2}\right)
+\left(r-\delta-\frac{f_n^2(y)} 2 +2\rho\sigma f_n(y)f_n'(y)\right) \frac{\partial }{\partial x} \\&\qquad +\left( \kappa\left(\theta-f_n^2(y)\right)+\sigma^2f_n(y)f'_n(y)\right)\frac{\partial }{\partial y}-2\kappa f_n(y)f_n'(y) + \zeta_\varepsilon'(u^n_\varepsilon-\psi)       -(r-\delta).
\end{align*}

By using the  Comparison principle, we deduce that, if $ f_n(y)f_n'(y) \left(\frac{\partial^2 u}{\partial x^2}-\frac{\partial u}{\partial x}\right)\geq 0 $, then $\bar u\geq 0$ and the assertion follows letting $\varepsilon$ tend to 0.

Since  $f_n$ is positive and  nondecreasing, it is enough to prove that $\frac{\partial^2 u}{\partial x^2}-\frac{  \partial u}{\partial x}\geq 0$. 
We write the equations satisfied by $u'=\frac{  \partial u}{\partial x}$ and $u''=\frac{  \partial^2 u}{\partial x^2}$. We have
\begin{equation}\label{u'}
\begin{cases}
-\frac{\partial u'}{\partial t}-\tilde{\mathcal{A}}^n u'+\zeta'_\varepsilon(u-\psi)(u'-\psi')=0,\\
u(T)=\psi,
\end{cases}
\end{equation} 
and
\begin{equation}\label{u''}
\begin{cases}
-\frac{\partial u''}{\partial t}-\tilde{\mathcal{A}}^nu''+\zeta''_\varepsilon(u-\psi)(u'-\psi')^2+\zeta'_\varepsilon(u-\psi)(u''-\psi'')=0,\\
u''(T)=\psi''.
\end{cases} 
\end{equation}
Using \eqref{u'} and \eqref{u''}, we get that $u''-u'$ satisfies
\begin{equation}
\begin{cases}
-\frac{\partial (u''-u')}{\partial t
	}-\mathcal{A}^n(u''-u')+\zeta'_\varepsilon(u-\psi)(u''-u')=\zeta'_\varepsilon(u-\psi)(\psi''-\psi')-\zeta''_\varepsilon(u-\psi)(u'-\psi')^2,\\
u''(T)-u'(T)=\psi''-\psi'.
\end{cases} 
\end{equation}
Recall that $\psi''-\psi'\geq0$ by assumption and
that $\zeta_\varepsilon$ is increasing and concave. Then, 
$$
\zeta'_\varepsilon(u-\psi)(\psi''-\psi')-\zeta''_\varepsilon(u-\psi)(u'-\psi')^2\geq 0, \quad u''(T)-u'(T)=\psi''-\psi'\geq0,
$$
hence, by using again the Comparison principle, we deduce that $ u''-u'\geq0$ which concludes the proof.
\end{proof}
The proof of Theorem \ref{monotonie} is now almost immediate.
\begin{proof}[Proof of Theorem \ref{monotonie}]
Thanks to Proposition \ref{mon2}, the function $u^n$ is increasing in the $y$ variable for all $n\in\N$. 	Then, the assertion follows by using Proposition \ref{convP}.
	\end{proof}
	
	\section{The American put price}\label{sect-put}
From now on we focus our attention on the standard put option with strike price $K$ and maturity $T$, that is we fix $\varphi(s)=(K-s)_+$ and we study the properties of the function
\begin{equation}\label{put-price}
P(t,s,y)=\sup_{\tau \in\mathcal T_{t,T}}\E[e^{-r(\tau-t )}(K-S^{t,s,y}_\tau)_+].
\end{equation} 
The following result easily follows from \eqref{put-price}.
\begin{proposition}\label{properties_P}
The price function $P$ satisfies:
	\begin{enumerate}
		\item $(t,s,y)\mapsto P(t,s,y)$ is continuous and positive;
		\item $t\mapsto P(t,s,y)$ is nonincreasing;
		\item $y\mapsto P(t,s,y)$ is nondecreasing;
		\item $s\mapsto P(t,s,y)$ is nonincreasing and convex.
	\end{enumerate}
\end{proposition}
\begin{proof}
	The proofs of $1.$ and $2.$ are classical and straightforward.  As regards $3.$, we
 note that $\varphi$ is convex and  the function $\psi(x)=(K-e^x)_+$ belongs to the space $W^{1,2}(\R,e^{-\gamma|x|})$ for a $\gamma>1$ but it is not regular enough to  apply  Proposition \ref{monotonie}. However, we can use an approximation procedure. Indeed, thanks to density results and \cite[Lemma 3.3]{JLL}, we can  approximate the function $\psi$ with a sequence of functions $\psi_n\in W^{2,2}(\R,e^{-\gamma|x|})\cap C^2(\R)$ such that $\psi_n''-\psi_n'\geq0$, so the assertion easily follows passing to the limit.   $4.$  follows from the fact that $\varphi(s)=(K-s)_+$ is nonincreasing and convex.
	\end{proof}

Moreover, thanks to  the Lipschitz continuity of the payoff function, we have the following result.
	\begin{proposition}	\label{prop-lipschitz}
		The function $x \mapsto u(t,x,y)$ is Lipschitz continuous while the function $y \mapsto u(t,x,y)$ is Holder continuous. If $2\kappa\theta \geq \sigma^2$ the function $y\mapsto u(t,x,y)$ is locally Lipschitz continuous on $(0,\infty)$. 
	\end{proposition}
	\begin{proof}
	It is easy to  prove that, for every fixed $t\geq 0$ and $y,y'\geq 0$ with $y\geq y'$, 
	\begin{equation}\label{flowfory}
	\E\left[Y^{y}_t-Y^{y'}_t\right] \leq y-y'. 
	\end{equation}
	Moreover, recall that $\dot{Y}^y_t \geq0$, so that $Y^{y}_t\geq Y^{y'}_t$ (see  \cite[Chapter 9,Theorem 3.7]{RY}). Therefore, \eqref{flowfory} can be rewritten as 
		\begin{equation}\label{flow-y}
		\E\left[|Y^{y}_t-Y^{y'}_t|\right] \leq |y-y'|. 
		\end{equation}
		Then, for $(x,y), (x',y')\in\R\times [0,\infty)$ we have
		\begin{align*}
	&	|u(t,x,y)-u(t,x',y')| = \left|    \sup_{\tau \in\mathcal{T}_{t,T}}\E[e^{-r(\tau-t )}(K-e^{X^{t,x,y}_\tau})_+]-  \sup_{\tau \in\mathcal{T}_{t,T}}\E[e^{-r(\tau-t )}(K-e^{X^{t,x',y'}_\tau})_+] \right|\\&\quad
		\leq     \sup_{\tau \in \mathcal{T}_{t,T}}\left|  \E\Big[e^{-r(\tau-t )}(K-e^{X^{t,x,y}_\tau})_+-e^{-r(\tau-t )}(K-e^{X^{t,x',y'}_\tau})_+  \Big]  \right| \\
		&\quad\leq C  \E \left[  \sup_{u \in [t,T]}|X_u^{t,x,y}      -X_u^{t,x',y'}| \right] \\
		&\quad \leq C\left( |x-x'|+\int_t^T\E [|Y^{t,y}_u-Y^{t,y'}_y|]du + \E\left[ \sup_{s\in[t,T]} \left|  \int_t^s (\sqrt{Y^{t,y}_u}-\sqrt{Y^{t,y'}_u})dW_u   \right| \right]\right)\\
		& \quad \leq C\left( |x-x'|+\int_t^T\E[ |Y^{t,y}_u-Y^{t,y'}_y|]du +\left(\E\left[ \sup_{s\in[t,T]} \left|  \int_t^s  (\sqrt{Y^{t,y}_u}-\sqrt{Y^{t,y'}_u})dW_u   \right| \right]^2\right)^{\frac 1 2} \right)\\
			&\quad \leq  C\left( |x-x'|+\int_t^T\E[|Y^{t,y}_u-Y^{t,y'}_u| ] du +\left(\E\left[  \int_t^T |Y^{t,y}_u-Y^{t,y'}_u|du    \right]\right)^{\frac 1 2} \right)\\
			&\quad \leq C_T(|x-x'|+\sqrt{|y-y'|}).
				\end{align*}
					Now, recall that, if $2\kappa \theta \geq \sigma^2$, the  volatility process $Y$ is strictly positive so we can apply It\^o's Lemma to the square root  function   and the process $Y_t$ in the open set $(0,\infty)$. We get
		\begin{align*}
		\sqrt{Y_t^y}&= \sqrt{y}+ \int_0^t \frac{1}{2\sqrt{Y_u^y}}dY^y_u - \frac 1 2 \int_0^t
		\frac{1}{4(Y^y_u)^{\frac 3 2 }}\sigma^2 Y_u^y du\\
		&=\sqrt{y}+ \left(  \frac{\kappa \theta}{2} - \frac{\sigma^2}{8} \right)  \int_0^t \frac{1}{\sqrt{Y_u^y}}du  - \frac \kappa 2 \int_0^t \sqrt{Y^y_u}du + \frac \sigma 2 W_t.
		\end{align*}
As already proved in \cite{Oa},	differentiating with respect to $y$,  one has
	\begin{equation}\label{stimaperholder}		\begin{split}
		\frac{\dot{Y_t^y}}{2\sqrt{Y_t^y}}  & = \frac{1}{2\sqrt{y}} + \left(  \frac{\kappa \theta}{2} - \frac{\sigma^2}{8} \right)  \int_0^t -\frac{\dot{Y_u^y}}{2(Y_u^y)^{\frac 3 2 }}du - \frac \kappa 2 \int_0^t \frac{\dot{Y_u^y}}{2\sqrt{Y_u^y}} du \leq \frac{1}{2\sqrt{y}}, \qquad a.s. 
		\end{split}
		\end{equation}
		since $\kappa\theta \geq \sigma^2/2\geq \sigma^2/4$ and $Y^y_t>0, \ \dot{Y}^y_t \geq0$. 
		
		Therefore, let us consider $y,y' \geq a$. Repeating the same calculations as before
		\begin{align*}
	&	|u(t,x,y)-u(t,x,y')| \\\quad
		&\quad\leq C\left(\int_t^T\E[|Y^{t,y}_u-Y^{t,y'}_u|]du +\left(\E\left[ \sup_{s\in[t,T]} \left|  \int_t^s  (\sqrt{Y^{t,y}_u}-\sqrt{Y^{t,y'}_u})dW_u   \right| \right]^2\right)^{\frac 1 2} \right)\\
	&	\quad\leq  C\left(\int_t^T\E[|Y^{t,y}_u-Y^{t,y'}_u|  ]du +\left(\E\left[  \int_t^T (\sqrt{Y^{t,y}_u}-\sqrt{Y^{t,y'}_u})^2du    \right]\right)^{\frac 1 2} \right)\\
		&\quad =  C\left(\int_t^T\E[|Y^{t,y}_s-Y^{t,y'}_s |]du +\left(\E\left[  \int_t^Tdu \left(\int_{y}^{y'}   \frac{\dot{Y}_u^{t,w}}{2\sqrt{Y_u^{t,w}}} dw  \right)^2  \right]\right)^{\frac 1 2} \right)\\
		&\quad \leq C_T\left(   |y-y'| + \left( \E\left[  \int_t^T \left(   \frac{1}{2\sqrt{a}} |y-y'|  \right)^2du  \right]\right)^{\frac 1 2}    \right)\\
		&\quad \leq C_T|y-y'|,
		\end{align*}
		which completes the proof. 
	\end{proof}
	\begin{remark}
Studying the properties of the put price also clarifies the behaviour of the call price since it is straightforward to extend to the Heston model the  symmetry relation between call and put prices. In fact, let us   highlight the dependence of the prices with respect to the  parameters $K,r,\delta,\rho$, that is let us write
$$
P(t,x,y;K,r,\delta,\rho)=\sup_{\tau\in\mathcal{T}_{t,T}}\E[e^{-r(\tau-t)}(K-S^{t,s,y}_\tau)_+],
$$
for the put option price and
$$ C(t,s,y;K,r,\delta,\rho)=\sup_{\tau\in\mathcal{T}_{t,T}}\E[e^{-r(\tau-t)}(S^{t,s,y}_\tau-K)_+],
$$
for the call option. Then, we have $C(t,s,y;K,r,\delta,\rho)=P(t,K,y;x,\delta,r,-\rho)$.

In fact, for every $\tau \in \mathcal{T}_{t,T}$, we have
\begin{align*}
&\E e^{-r(\tau-t)}  \bigg( se^{\int_t^\tau\left( r-\delta -\frac{Y^{t,y}_s}2\right)ds+\int_t^\tau \sqrt{Y^{t,y}_s}dB_s}-K\bigg)_+  \\&\qquad=\E e^{-\delta(\tau-t)}e^{\int_t^\tau \sqrt{Y^{t,y}_s}dB_s-\int_t^\tau\frac{Y^{t,y}_s}2ds } \bigg(  x-Ke^{\int_t^\tau\left(\delta-r +\frac{Y^{t,y}_s}2\right)ds -\int_t^\tau dB_s  } \bigg)_+  \\&\qquad=\E e^{-\delta(\tau-t)}e^{\int_t^T \sqrt{Y^{t,y}_s}dB_s-\int_t^T\frac{Y^{t,y}_s}2ds } \bigg(  x-Ke^{\int_t^\tau\left(\delta-r +\frac{Y^{t,y}_s}2\right)ds -\int_t^\tau  \sqrt{Y^{t,y}_s}dB_s  } \bigg)_+ , 
\end{align*}
where the last equality follows from the fact that $(e^{\int_t^s \sqrt{Y^{t,y}_s}dB_s-\int_t^s\frac{Y^{t,y}_s}2ds })_{s\in[t,T]}$ is a martingale. Then, note that the process  $\hat B_t= B_t-\sqrt{Y^{t,y}_t }t$ is a Brownian motion under the  probability measure $\hat P$ which has density $d\hat{\P}/d\P=e^{\int_t^T \sqrt{Y^{t,y}_s}dB_s-\int_t^T\frac{Y^{t,y}_s}2ds }$. Therefore
$$
\E e^{-r(\tau-t)}  \bigg( se^{\int_t^\tau\left( r-\delta -\frac{Y^{t,y}_s}2\right)ds+\int_t^\tau \sqrt{Y^{t,y}_s}dB_s}-K\bigg)_+ =\hat \E e^{-\delta (\tau-t)} \bigg(  x-Ke^{\int_t^\tau\left(\delta-r -\frac{Y^{t,y}_s}2\right)ds -\int_t^\tau \sqrt{Y^{t,y}_s}dB_s  } \bigg)_+.
$$
Under the probability $\hat{\P}$, the process $(-\hat B,W )$ is a Brownian motion with correlation coefficient $-\rho$ so that the assertion follows.
	\end{remark}
	 \subsection{The  exercise boundary}
	 Let us introduce the so called continuation region
	 $$
\mathcal C=\{ (t,s,y) \in [0,T)\times (0,\infty)\times[0,\infty) :P(t,s,y)>\varphi(s)  \}
	 $$
	 and its complement, the exercise region
	 $$
	 \mathcal E= \mathcal C^c=\{ (t,s,y) \in [0,T)\times (0,\infty)\times[0,\infty) :P(t,s,y)=\varphi(s)\}.
	 $$
	 Note that, since $P$ and $\varphi$ are both continuous,  $\mathcal C$ is an (relative) open set while  $\mathcal E$ is a closed set.

	 Generalizing the standard definition given in the Black and Scholes  type models,  we consider the \textit{critical exercise price} or \textit{free exercise boundary}, defined as
	 	$$
	 	b(t,y)=\inf\{ s>0| P(t,s,y)>(K-s)_+\},\qquad (t,y)\in [0,T)\times [0,\infty).
	 	$$
	 	We have $P(t,s,y)=\varphi(s)$  for $s\in [0,b(t,y))$ and also for $s= b(t,y)$, due to the continuity of $P$ and $\varphi$.	
	 	Note also that,  since $P>0$, we have $b(t,y) \in [0,K)$. 
	 	Moreover, since $P$ is convex, we can write
	 	$$
	 	\mathcal C=\{(t,s,y) \in [0,T)\times (0,\infty)\times[0,\infty) :  s>b(t,y)    \}
	 	$$
	 	and
	 	$$
	 	\mathcal E=\{(t,s,y) \in [0,T)\times (0,\infty)\times[0,\infty) :  s\leq b(t,y)    \}.
	 	$$


 We now study some properties of the free boundary  $b:[0,T)\times [0,\infty)\rightarrow [0,K)$.  First of all, we have the following simple result.
	 \begin{proposition}
	 	We have:
	 	\begin{enumerate}
	 		\item 	for every fixed $y\in[0,\infty)$, the function $t\mapsto b(t,y)$ is nondecreasing and right continuous;
	 		\item  for every fixed $t\in[0,T)$, the function $y\mapsto b(t,y)$ is nonincreasing and  left continuous.
	 	\end{enumerate}
	 \end{proposition}
	 \begin{proof}
	 	$1.$ Recalling that the map $t\mapsto P(t,s,y)$ is nonincreasing, we directly deduce that $t\mapsto b(t,y)$  is nondecreasing.
	 	Then, fix $t\in [0,T)$ and let $(t_n)_{n \geq 1}$ be a decreasing sequence such that $\lim_{n \rightarrow \infty} t_n = t$. The sequence $(b(t_n,y))_n$ is nondecreasing so that $ \lim_{n \rightarrow \infty}b(t_n,y)$ exists and we have $  \lim_{n \rightarrow \infty}b(t_n,y) \geq b(t,y)$.
	 		 	On the other hand, we have		
	 	\begin{equation*}
	 	P(t_n,b(t_n,y),y)= \varphi(b(t_n,y)) \qquad  n \geq 1,
	 	\end{equation*}
	 	and, by the continuity of $P$ and $\varphi$,
	 	\begin{equation*}
	 	P(t,\lim_{n \rightarrow \infty}b(t_n,y),y)= \varphi(\lim_{n \rightarrow \infty}b(t_n,y)).
	 	\end{equation*}
	 	We deduce by the definition of $b$ that $  \lim_{n \rightarrow \infty}b(t_n,y) \leq b(t,y)$ which concludes the proof.
	 	
	 	$2.$  The second assertion can be proved with the same arguments, this time recalling that $y\mapsto P(t,s,y)$ is a nondecreasing function.
	 \end{proof}
	 Recall that $b(t,y) \in [0,K)$.   Indeed, we can prove  the positivity of the function.
		\begin{proposition}\label{positivity}
			 We have $b(t,y)>0$ for every $(t,y) \in [0,T)\times [0,\infty)$.
		\end{proposition}
		\begin{proof}
			Without loss of generality we can assume that $0<t<T$, since $T$ is arbitrary and the put price is a function of $T-t$.  Suppose that $b(t^*,y^*)=0$ for some $(t^*,y^*) \in (0,T)\times [0,	\infty) $. Since $b(t,y)\geq 0$, $t \mapsto b(t,y)$ is nondecreasing and $y \mapsto b(t,y)$ is nonincreasing, we have $ b(t,y)=0$ for $(t,y) \in (0,t^*) \times (y^*,\infty)$, so that
			\begin{equation*}
			 P(t,s,y)> \varphi(s),\qquad (t,s,y) \in (0,t^*) \times(0,\infty) \times (y^*,\infty).
			\end{equation*}
		To simplify the calculations, we pass  to the logarithm in the space variable and we consider the functions $u(t,x,y)= P(t,e^x,y)$ and $\psi(x)=\varphi(e^x) $. We  have $u(t,x,y)>\psi(x)$ and $$(\partial_t+\tilde{\mathcal{L}}-r)u=0 \qquad \mbox{ on } (0,t^*) \times \R \times (y^*,\infty),$$ where $\tilde \L$ was defined  in \eqref{Llog}. Since $t \mapsto u(t,x,y)$ is nondecreasing, we deduce that, for $t\in (0,t^*)$, $(\tilde \L-r)u=-\partial_t u\geq 0$ in the sense of distributions. Therefore, for any nonnegative and $C^\infty$ test functions $\theta$, $\phi$ and $\zeta$ which have support respectively in $(0,t^*)$, $(-\infty,\infty)$ and $(y^*,\infty)$, we have 
				\begin{equation*}
			\int_0^{t^*}\theta(t)dt \int_{-\infty}^\infty dx \int_{y^*}^\infty  \tilde{ \mathcal{L} }u(t,x,y) \phi(x)\zeta(y)dy \geq r 	\int_0^{t^*}\theta(t)dt \int_{-\infty}^\infty dx \int_{y^*}^\infty  (K-e^x)\phi(x)\zeta(y)dy,
			\end{equation*}
		or equivalently, by the continuity of the integrands in $t$, 			
			\begin{equation} \label{distr}
			\int_{-\infty}^\infty dx \int_{y^*}^\infty \tilde {\mathcal{L}} u(t,x,y) \phi(x)\zeta(y) dy \geq r 	 \int_{-\infty}^\infty dx \int_{y^*}^\infty  (K-e^x)\phi(x)\zeta(y) dy.
			\end{equation}
			Let $\chi_1$ and $\chi_2 $ be two nonnegative $C^\infty$ functions  such that $\supp\chi_1\subseteq [-1,0]$, $\supp\chi_2\subseteq [0,1]$ and $\int \chi_1(x)dx=\int \chi_2(x)dx=1$. Let us apply \eqref{distr} with $\phi(x)=\lambda \chi_1(\lambda x)$ and $\zeta(y)=\sqrt{\lambda} \chi_2(\sqrt{\lambda} (y-y^*))$, with $\lambda>0$. We will prove in a moment that
		\begin{equation}\label{pezzo1}
			\limsup_{\lambda \downarrow 0}	\int_{-\infty}^\infty dx \int_{y^*}^\infty \tilde {\mathcal{L}} u(t,x,y) \lambda \chi_1(\lambda x)\sqrt{\lambda} \chi_2(\sqrt{\lambda} (y-y^*))dy\leq 0
		\end{equation}
			and 
			\begin{equation}\label{pezzo2}
			\lim_{\lambda \downarrow 0}  r 	 \int_{-\infty}^\infty dx \int_{y^*}^\infty (K-e^x) \lambda \chi_1(\lambda x)\sqrt{\lambda} \chi_2(\sqrt{\lambda} (y-y^*))dy =rK>0,
			\end{equation}
			which contradicts \eqref{distr}, concluding the proof. 
			
			 As regards \eqref{pezzo2}, we have
				\begin{align*}
			r 	 \int_{-\infty}^\infty dx \int_{y^*}^\infty (K-e^x) \lambda \chi_1(\lambda x)\sqrt{\lambda} \chi_2(\sqrt{\lambda} (y-y^*))dy= rK -r\int_{-\infty}^\infty e^{\frac{x}{\lambda}} \chi_1(x)dx .
			\end{align*}
Since $ \supp\chi_1 \subset [-1,0]$, $\lim_{\lambda\rightarrow 0}\int e^{\frac{x}{\lambda}} \chi_1(x)dx =0$, so that
			\begin{equation}\label{leftside}
			\lim_{\lambda\rightarrow 0} r  \int_\R dx \int_{-\infty}^{y^*} (K-e^x)\phi(x)\zeta(y)dy= rK >0.
			\end{equation}
	Concerning  \eqref{pezzo1}, we can write
		\begin{align*}
		&\int_{-\infty}^{+\infty} dx \int_{y^*}^{\infty}\tilde{\mathcal{L}   } u(t,x,y) \phi(x)\zeta(y)dy\\&=\int_{-\infty}^{+\infty} dx \int_{y^*}^{\infty}\tilde{\mathcal{L}_0   } u(t,x,y) \phi(x)\zeta(y)dy    -   	\int_{-\infty}^{+\infty} dx \int_{y^*}^{\infty}Ky\frac{\partial }{\partial y}u(t,x,y) \phi(x)\zeta(y)dy ,
			\end{align*}
			where
			$$
		\tilde{\mathcal{L}}_0 =\frac y 2 \left( \frac{\partial^2}{\partial x^2}+2\rho\sigma \frac{\partial^2}{\partial x \partial y}+ \sigma^2\frac{\partial^2}{\partial y^2} \right)+\left( r-\delta-\frac y 2 \right)\frac{\partial }{\partial x}+\kappa\theta\frac{\partial }{\partial y}.
			$$
We can easily  prove that $\lim_{\lambda\downarrow 0}   \int_{-\infty}^{+\infty} dx \int_{y^*}^{\infty}\tilde{\mathcal{L}_0   } u(t,x,y) \phi(x)\zeta(y)dy   =0$.  For example,  integrating by parts two times, we have
			\begin{align*}
		&	\int_{-\infty}^{+\infty} dx \int_{y^*}^\infty   \frac y 2 \sigma^2 \frac{\partial^2}{\partial  y^2} u(t,x,y)    \lambda \chi_1(\lambda x)\sqrt \lambda \chi_2(\sqrt \lambda (y-y^*)) dy\\
		&\quad=-	\int_{-\infty}^{+\infty} dx \int_{y^*}^\infty  \frac  {\sigma^2} 2 \frac{\partial}{\partial  y} u(t,x,y)    \lambda \chi_1(\lambda x)\left( \sqrt \lambda \chi_2(\sqrt \lambda (y-y^*)) + y\lambda \chi'_2(\sqrt \lambda (y-y^*)) \right)dy\\
		&\quad =\int_{-\infty}^{+\infty} dx \int_{y^*}^\infty    \frac  {\sigma^2} 2 u(t,x,y)  \left( 2 \lambda \chi_1(\lambda x)\lambda \chi'_2(\sqrt \lambda (y-y^*))\right)dy \\
		&\quad=\sqrt \lambda  \sigma^2\int_{-\infty}^{+\infty} dx \int_0^{\infty }   u\left(t,\frac x \lambda ,\frac y {\sqrt \lambda}+y^*\right)    \chi_1( x) \left(  \lambda \chi'_2(y)+\frac 1 2\lambda ^{\frac 3 2} \left( y +\sqrt \lambda y^*\right) \chi''_2(y)\right)  dy
			\end{align*}
which  tends to 0 as $\lambda$ goes to 0 since $u$ is bounded. The other terms in $\tilde \L_0$ can be treated with similar arguments.
On the other hand, we have
	$$
	-\int_{-\infty}^{+\infty} dx \int_{y^*}^{\infty}\kappa y\frac{\partial }{\partial y}u(t,x,y)\lambda \chi_1(\lambda x )\sqrt \lambda\chi_2(\sqrt \lambda(y-y^*))dy\leq 0
	$$
for any $\lambda>0$,	since $u$ is nondecreasing in $y$. Therefore  \eqref{pezzo1} is proved.
	\end{proof}
As regards  the regularity of the free boundary, we can prove the following result. 

	\begin{proposition}
		For any $t\in[0,T)$ there exists a countable set $\mathcal N\subseteq (0,\infty)$ such that
		$$
	b(t^-,y)= b(t,y), \qquad y\in (0,\infty)\setminus\mathcal N.
		$$
	\end{proposition}
	\begin{proof}
		Without loss of generality we pass to the logarithm in the $s-$variable and we prove the assertion for the function $\tilde b(t,y)=\ln b(t,y)$.
		Fix $t\in [0,T)$ and recall that $y\mapsto \tilde b (t,y)$ is a nonincreasing function, so it has at most a countable set of discontinuity points. Let $y^*\in (0,\infty)$ be a continuity point for the maps $y\mapsto \tilde b(t,y)$ and $y\mapsto \tilde b(t^-,y)$ and assume that
		\begin{equation}\label{assu}
		\tilde b(t^-,y^*)<\tilde b(t,y^*).
		\end{equation}
		Set $	\epsilon=\frac {\tilde b(t,y^*)-\tilde b(t^-,y^*) }2$.  By continuity, there exist $y_0,y_1> 0$ such that for any $y\in (y_0,y_1)$ we have
		$$
		\tilde b(t,y)>	\tilde b(t,y^*)-\frac \epsilon 4, \qquad \mbox{ and }\qquad
		\tilde b(t^-,y)<	\tilde b(t^-,y^*)+\frac \epsilon 4.
		$$
		Therefore, by using \eqref{assu}, we get, for any $y\in (y_0,y_1)$,
		$$
		\tilde b(t,y)>	\tilde b(t,y^*)-\frac \epsilon 4>\tilde b(t^-,y^*)+\frac 3 4 \epsilon >	\tilde b(t^-,y^*)+\frac \epsilon 4>\tilde b(t^-,y).
		$$
		Now, set $b^-=\tilde b(t^-,y^*)+\frac \epsilon 4$ and $b^+=\tilde b(t^-,y^*)+\frac 3 4 $ and let  $(s,x,y)\in (0,t)\times (b^-,b^+)\times (y_0,y_1)$. Since  $t\mapsto		\tilde b(t,\cdot)$ is nondecreasing, we have $x>\tilde b(t^-,y)\geq \tilde b(s,y)$, so that $u(s,x,y)>\psi(x)$. Therefore, on the set $(0,t)\times (b^-,b^+)\times (y_0,y_1)$ we have 
		$$
		(	\tilde \L-r)u(s,x,y)= -\frac{\partial u}{\partial t}(s,x,y)\geq 0.
		$$
		This means that, for any nonnegative and $C^\infty$ test functions $\theta$, $\psi$ and $\zeta$ which have  support respectively in $(0,t)$, $(b^-,b^+)$ and $(y_0,y_1)$ we can write
				\begin{equation*}
				\int_0^{t}\theta(\tau)d\tau \int_{-\infty}^\infty dx \int_{y^*}^\infty dy ({\tilde{\L}} -r)u(t,x,y) \phi(x)\zeta(y) \geq 0.
				\end{equation*}
				By the continuity of the integrands in $t$, 	we deduce that 		$(\tilde \L -r)u(t,\cdot,\cdot)\geq 0$ in the sense of distributions on the set $(b^-,b^+)\times (y_0,y_1)$.

	 On the other hand,  for any $(s,x,y)\in (t,T)\times (b^-,b^+)\times (y_0,y_1)$,  we have $x\leq	\tilde b(t,y)\leq 	\tilde b(s,y)$, so that $u(s,x,y)=\psi(x)$. Therefore, it follows from   $\frac{\partial u}{\partial t}+(	\tilde \L-r)u\leq 0$ and the continuity of the integrands that $(	\tilde \L-r)u(t\cdot, \cdot)=(	\tilde \L-r)\psi(\cdot)
	 \leq 0$ in the sense of distributions on the set $ (b^-,b^+)\times (y_0,y_1)$. 
	 
	 We deduce that $(	\tilde \L-r)\psi= 0$ on the set $(b^-,b^+)\times (y_0,y_1)$, but it is easy to see that  $
	 (	\tilde \L-r)\psi(x)=(	\tilde \L-r)(K-e^x)=\delta e^x-rK$ and thus cannot be identically zero in a nonempty open set.
	\end{proof}

\begin{remark}
It is worth observing that the arguments used  in \cite{V} in order to prove the continuity of the exercise price of American options in a multidimensional Black and Scholes model can be easily adapted to our framework.  In particular, if  we  consider the $t$-sections of the exercise region, that is
\begin{equation}\label{tsection}
\begin{split}
\mathcal E_t&= \{ (s,y)\in (0,\infty)\times [0,\infty): P(t,s,y)=\varphi(s)      \},	\\&=\{ (s,y)\in (0,\infty)\times [0,\infty): s\leq b(t,y)\}, \qquad\qquad\qquad t\in [0,T),
\end{split}
\end{equation}
 we can easily prove that 
 	\begin{equation}\label{tsections}
\Ex_t=\bigcap_{u>t}\Ex_u,\qquad\qquad
 \Ex_t=\overline{\bigcup_{u<t}\Ex_u}.
 	\end{equation}
 	However, unlike the case of an American option on several assets, in our case \eqref{tsections} is not sufficient   to  deduce the continuity of the function $t\mapsto b(t,y)$.
\end{remark}
		 \subsection{Strict convexity in the continuation region}
		 We know that $P$ is convex in the space variable (see Proposition \ref{properties_P}). In \cite{T} it is also proved that, in the case of non-degenerate stochastic volatility models, $P$ is strictly convex in the continuation region but the proof follows an analytical approach which cannot be applied in our degenerate model.   In this section we extend this result to the Heston model by using  purely probabilistic techniques. 
		 
		 We will need the following Lemma, whose proof can be found in the Appendix.
		  \begin{lemma}\label{support}
For every continuous function $s: [0,T]\rightarrow (0,\infty)$ such that $s(0)=S_0$ and for every $\epsilon>0$ we have
$$
\P\left(\sup_{t\in [0,T]}|S_t-  s (t)|<\epsilon , \sup_{t\in [0,T]}|Y_t-Y_0|<\epsilon \right)>0.
$$
 \end{lemma}
		
		 \begin{theorem}
		 	The function $s \mapsto P(t,s,y)$ is strictly convex in the continuation region.
		 \end{theorem}
		 \begin{proof}
		 	Without loss of generality we can assume $t=0$. 	 We have to prove that, if $(s_1,y), \, (s_2,y) \in (0,\infty)\times [0,\infty)$ are such that  $(0,s_1,y), (0,s_2,y) \in \mathcal{C}$, then
		 	\begin{equation}\label{conv}
		 	P(0,\theta s_1+(1-\theta)s_2,y) < 
		 	\theta P(0, s_1,y)+(1-\theta) P(0,s_2,y).
		 	\end{equation}
		 	Let us   rewrite the price process  as
$
		 	S_t^{s,y}=se^{ \int_{0}^{t}\left(r-\delta-\frac {Y_u}{2}\right)du +\int_{0}^{t}\sigma \sqrt{Y_u}dB_u  }:= sM^y_t,
$
		 	where $M^y_t=S^{1,y}_t$ and  assume that, for example, $s_1>s_2$.
		 	We claim that it is enough to prove that, for 
		 	 $\varepsilon>0$ small enough,
		 	\begin{equation}\label{conv_2}
		 	\begin{split}
		 	\P\Big(  &(\theta s_1+(1-\theta)s_2)M^y_t>b(t,Y_t) \, \forall t \in [0,T)\, \& \,(\theta s_1+(1-\theta)s_2)M^y_T \in (K-\varepsilon,K+\varepsilon) \Big) >0.
		 	\end{split}
		 	\end{equation}
		 	In fact, 	let $\tau^*$ be the optimal stopping time for $P(0,\theta s_1+(1-\theta)s_2,y)$. If  $ (\theta s_1+(1-\theta)s_2)M^y_t>b(t,Y_t)$ for every $ t \in [0,T) $, then we are in the continuation region for all $  t \in [0,T) $, hence  $\tau^*=T$. Then, the condition $ (\theta s_1+(1-\theta)s_2)M^y_T \in (K-\varepsilon,K+\varepsilon) $ for  $\varepsilon>0$ small enough ensures on one hand that $s_1M^y_{\tau^*}>K$, since 
		 	\begin{align*}
		 	s_1M^y_{\tau^*}	& = (\theta s_1+(1-\theta)s_2 )M^y_{\tau^*}+ (1-\theta)(s_1-s_2)M^y_{\tau^*}\\
		 	& >K-\varepsilon + \frac{(1-\theta)(s_1-s_2)(K-\varepsilon)}{\theta s_1+(1-\theta)s_2 }>K,
		 	\end{align*}
		 	for $\varepsilon$ small enough. 		 	On the other hand, it also ensures that $ s_2M^y_{\tau^*}<K $, which can be proved  with similar arguments. 
		 		 	Therefore, we get
		 	\begin{equation*}\label{condition_convexity}
		 	\P\left((K-s_1M^y_{\tau^*})_+=0\, \& \, (K-s_2M^y_{\tau^*})_+>0 \right)>0,
		 	\end{equation*}
		  which, from a closer look  at the graph of the function $x	\mapsto (K-x)_+$,  implies that
		 	\begin{align*}
		 	\E[e^{-r\tau^*}(K-(\theta s_1+(1-\theta)s_2 ) M^y_{\tau^*})_+] < \theta \E[e^{-r\tau^*}(K-s_1 M^y_{\tau^*})_+] +(1-\theta)\E[e^{-r\tau^*}(K-s_2 M^y_{\tau^*})_+],
		 	\end{align*}
		 	and, as a consequence, \eqref{conv}.

	So, the rest of the proof is devoted  to prove that \eqref{conv_2} is actually  satisfied.

With this aim, we first consider a suitable continuous function $m:[0,T]\rightarrow\R$ constructed as follows. 
	 In order to simplify the notation, we set $s=\theta s_1+(1-\theta)s_2$. 
		 Note that, for $\varepsilon>0$ small enough, we have $s=\theta s_1+(1-\theta)s_2 >b(0,y)+\varepsilon$ since $(0,s_1,y)$ and $ (0,s_2,y)$ are in the continuation region $ \mathcal{C}$, that is $s_1,s_2\in (b(0,y),\infty)$.  By the right continuity of the map $t\mapsto b(t,y)$, we know that there exists $\bar t\in (0,T)$ such that $s>b(t,y)+\frac \varepsilon 2$ for any $t\in [0,\bar t]$. Moreover the function $y\mapsto b(\bar t, y)$ is left continuous and nonincreasing, so there exists $\eta_\varepsilon>0$ such that $s >b(\bar t,z)+\frac \varepsilon 4$ for any $z\geq y-\eta_\varepsilon$.  Assume now that $s\leq K+\frac \varepsilon 2$ and
 set 
	 $$
	 m(t)=\begin{cases}1 + \frac t {\bar t} \left( \frac{ K+\frac \varepsilon 2 }{s}- 1 \right), \qquad &0\leq t\leq \bar t,\\
	 \frac{K+\frac \varepsilon 2}{s},\qquad & \bar t\leq t\leq T.
	 \end{cases}
	 $$
	Note that $m$ is continuous, $m(0)=1$ and, recalling that $t\mapsto b(t,y)$ is nondecreasing and $b(t,y)<K$,
	$$
	sm(t)=\begin{cases}
s + \frac t {\bar t} \left(K+\frac \varepsilon 2 - s \right)	\geq s >b(\bar t,y-\eta_\varepsilon )+\frac \varepsilon 4\geq b( t,y-\eta_\varepsilon )+\frac \varepsilon 4, \qquad &0\leq t\leq \bar t,\\
	K+\frac \varepsilon 2\geq b(t,y-\eta_\varepsilon),  &\bar t\leq  t\leq T.
	\end{cases}
	$$
	 Moreover, by  Lemma \ref{support}, we know that, for any $\epsilon>0$,
$$
\P\left(\sup_{t\in [0,T]}|sM^y_t-  sm (t)|<\epsilon , \sup_{t\in [0,T]}|Y_t-y|<\epsilon \right)>0.
$$
Therefore, by applying Lemma \ref{support} with  $\epsilon=\min\left\{\frac \varepsilon 8,\eta_\varepsilon\right\}$, we have that, with positive probability,
$$
sM^y_t>sm(t)-\frac \varepsilon 8\geq b(t, y-\eta_\varepsilon)+\frac \varepsilon 8\geq b(t,Y_t).
$$
	and
	$$
	sM^y_T\leq sm(T)+\frac \varepsilon 8\leq K+\varepsilon,\qquad 	sM^y_T\geq sm(T)-\frac \varepsilon 8\geq K-\varepsilon,
	$$
which proves \eqref{conv_2} concluding the proof.  If $s>K+\frac \varepsilon 2$, then it is enough to take $m(t)$ as a nonincreasing continuous function such that $m(0)=1$ and $sm(T)=K+\frac \varepsilon 2$. Then, the assertion follows with the same reasoning.

%
%
		 \end{proof}
		 
			\subsection{Early exercise premium}\label{sect-EEP}
We now extend to the stochastic volatility	Heston model a well known result in the Black and Scholes world, the so called \textit{early exercise premium formula}. It is an explicit formulation of the quantity $P-P_e$, where $P_e=P_e(t,s,y)$ is the European put price with the same strike price $K$ and maturity $T$ of the American option with price function $P=P(t,s,y)$.  Therefore, it represents the additional price you have to pay for the possibility of exercising before maturity.
			\begin{proposition}\label{EEP}
			Let $P_e(0,S_0,Y_0)$ be the European put price at time $0$ with maturity $T$ and strike price $K$. Then, one has
				\begin{equation*}
				P(0,S_0,Y_0)=P_e(0,S_0,Y_0)-\int_0^T e^{-rs}\E[(\delta S_s  -rK)\mathbf{1}_{\{S_s\leq b(s,Y_s)\}}]ds.
				\end{equation*}
			\end{proposition}

		The proof of Proposition \ref{EEP} relies on purely probabilistic  techniques and is based on the  results first  introduced in \cite{J}. Let $U_t=e^{-rt}P(t,S_t,Y_t)$ and $Z_t=e^{-rt}\varphi(S_t)$. Since $U_t$ is a supermartingale, we have the Snell decomposition 
		\begin{align}\label{doob_dec}
		& U_t=M_t-A_t,
		\end{align}
		where $M$ is a martingale and $A$ is a nondecreasing predictable process with $A_0=0$, continuous with probability 1 thanks to the continuity of $\varphi$. On the other hand, 
			\begin{align*}
			Z_t=e^{-rt}(K-S_t)_+&=Z_0-r\int_0^t e^{-rs}(K-S_s)_+ds -\int_0^t e^{-rs}\mathbf{1}_{\{ S_s\leq K\}}dS_s +\int_0^te^{-rs}dL^K_s(S)\\
			&=m_t+a_t,
			\end{align*}
		where $L^K_t(S)$ is the local time of $S$ in $K$,
		$$
		m_t=Z_0 -\int_0^te^{-rs}\mathbf{1}_{ \{ S_s\leq K\}  }S_s\sqrt{Y_s}dB_s
		$$  is a local martingale, and 
			$$
			a_t=-r\int_0^t e^{-rs}(K-S_s)_+ds -\int_0^t e^{-rs}\mathbf{1}_{(-\infty,K]}S_s(r-\delta)ds +\int_0^te^{-rs}dL^K_s(S)
			$$
			 is a predictable process with finite variation and $a_0=0$. Recall that $a_t$ can be written as the sum of an increasing and a decreasing component, that is $a_t=a_t^++a_t^-.$
				Since $(L^K_t)_t$ is increasing, we deduce that the decreasing process $(a_t^-)_t$ is absolutely continuous with respect to the Lebesgue measure, that is 
			$$
	da_t^- \ll dt.
			$$
		
We now define $$\zeta_t=U_t-Z_t \geq 0.$$
		Thanks to  Tanaka's formula,
		\begin{equation*}\label{Tanaka}
		\zeta_t=	\zeta_t^+= \zeta_0 + \int_0^t \mathbf{1}_{\{\zeta_s >0\}}d\zeta_s+ \frac 12L^0_t(\zeta),
		\end{equation*}	
		where $L^0_t(\zeta)$ is the local time of $\zeta$ in $0$.
		Therefore,
		\begin{align*}
		\zeta_t&= \zeta_0 + \int_0^t \mathbf{1}_{\{\zeta_s >0\}}d(U_s-Z_s)+ \frac 12L^0_t(\zeta)\\
		&=\zeta_0 + \int_0^t \mathbf{1}_{\{\zeta_s >0\}}dM_s- \int_0^t \mathbf{1}_{\{\zeta_s >0\}}dm_s - \int_0^t \mathbf{1}_{\{\zeta_s >0\}}da_s+ \frac 12L^0_t(\zeta),
		\end{align*}	
		where the last equality follows from the fact that the process $A_t$ only increases on the set $\{\zeta_t=0\}$.
		Then, we can write
		\begin{align*}
		U_t&= U_0 + \bar{M}_t - \int_0^t \mathbf{1}_{\{\zeta_s >0\}}da_s+ \frac 12L^0_t(\zeta)+a_t  =U_0 + \bar{M}_t + \int_0^t \mathbf{1}_{\{\zeta_s =0\}}da_s+ \frac 12L^0_t(\zeta),
		\end{align*}
		where $\bar{M}_t=\int_0^t \mathbf{1}_{\{\zeta_s>0\}}d(M_s-m_s)+m_t$ is a local martingale. 
		Thanks to the continuity of $U_t$ we have the uniqueness of the decompositions, so
		\begin{equation}\label{A}
		-A_t= \int_0^t \mathbf{1}_{\{\zeta_s =0\}}da_s+ \frac 12L^0_t(\zeta).
		\end{equation}
		This means in particular that $\int_0^t \mathbf{1}_{\{\zeta_s =0\}}da_s+ \frac 12L^0_t(\zeta)$ is decreasing, but $L^0_t(\zeta)$ is increasing  so $-\int_0^t \mathbf{1}_{\{\zeta_s =0\}}da_s$ must be an increasing process and 
		\begin{equation*}
		\frac 12 dL^0_t(\zeta) \ll \mathbf{1}_{\{\zeta_t =0\}}da_t^-\ll dt.
		\end{equation*}
		We define $\mu_t$ the density of $ L^0_t(\zeta) $ w.r.t. $dt$. Note that, by Motoo Theorem (see \cite{DM}),  we can write $\mu_t=\mu(t,S_t,Y_t)$. 

 Now, let us prove the following preliminary result.
		\begin{lemma}\label{local_time}
			The local time $L^0_t(\zeta)$ is indistinguishable from 0. 
		\end{lemma}
		\begin{proof}
			First of all, note that $L^0_t(\zeta)$ only increases on the set $\{(t,S_t,Y_t)\in \partial\mathcal{E}\}$.
			 In fact, recall that  $L^a_t=\int_0^t\mathbf{1}_{\{U_s-Z_s=a\}}dL^a_s$  for every $a>0$ and $t>0$, so that		 
		$$
			\int_0^t\mathbf{1}_{\{(s,S_s,Y_s) \in \mathring{\mathcal{E}}\}}dL^a_s=0.
$$
Moreover it is well known that, for any $t>0$, $L^0_t=\lim_{a\rightarrow 0}L^a_t$, which implies that the measures $L^a_t$ weakly converge to $L^0_t$ as $a\rightarrow 0$. Then, we can deduce that 
$$
\int_0^T \mathbf{1}_{\{(s,S_s,Y_s) \in \mathring{\mathcal{E}}\}}dL^0_s\leq \liminf	\int_0^T\mathbf{1}_{\{(s,S_s,Y_s) \in \mathring{\mathcal{E}}\}}dL^a_s=0
$$
Therefore, we have 
$$
	\E[L^0_t(\zeta)]= \E\left[\int_0^t\mathbf{1}_{\{U_s-Z_s=0\}}dL^0_s]\right]= \E\left[\int_0^t\mathbf{1}_{\{S_s\leq b(s,Y_s)\}}\mu(s,S_s,Y_s)k(s,S_s,Y_s)ds\right]
$$
Recall that  $t\mapsto b(t,y)$ is nondecreasing and right continuous, while  $y\mapsto b(t,y)$ is nonincreasing and left continuous. Let us consider the function 
 $$
 \hat b(t,y)=\sup_{s<t,y>y} b(s,z)
 $$
and the open set $\{  (t,x,y)\mid x<\hat b(t,y) \}\subseteq[0,T]\times \R\times [0,\infty)$. Thanks to the continuity of the trajectories, we have that the set $\{t\in [0,T]\mid  S_t<\hat{b}(t,Y_t)  \}$ is an open set, so that
$$
\int_0^T \mathbf{1}_{\{   S_t<\hat b(t,Y_t) \}}dL^0_t=0.
$$
Therefore, in order to prove that $	\E[L^0_t(\zeta)]=0$, it suffices to prove that
$$
\E\left[\int_0^T \mathbf{1}_{\{  \hat b(t,Y_t)\leq  S_t\leq b(t,Y_t) \}}dL^0_t\right]=0.
$$
Since the pair $(S_t,Y_t)$ has density, it is enough to prove that
\begin{equation}\label{lasttoprove}
\int_0^T dt\int  \mathbf{1}_{\{  \hat b(t,y)\leq  x\leq b(t,y) \}}dxdy=\int_0^T \int (b(t,y)-\hat b(t,y))dy=0
\end{equation}
In order to prove \eqref{lasttoprove}, note that $\hat b(t,y)=\lim_{n\rightarrow}b(t-\frac 1 n,y+\frac 1 n)$. Then, recalling that $\hat b\leq b$, for any $A>0$, with a simple change of variables we get
$$
\int_0^Tdt\int_0^A b(t,y)dy=\lim_{n\rightarrow\infty}\int_0^{T-\frac 1  n}dt\int_0^{A+\frac 1 n} b(t,y)dy, 
$$
from which we deduce \eqref{lasttoprove}.
		\end{proof}
We can now prove Proposition \ref{EEP}. 
\begin{proof}[Proof of Proposition \ref{EEP}]
		Thanks to \eqref{A} and Lemma \ref{local_time}  we can rewrite \eqref{doob_dec} as
		\begin{align*}
		U_t&=M_t + \int_0^t \mathbf{1}_{\{U_s=Z_s\}}da_s=M_t+ \int_0^t e^{-rs}(\mathcal{L}-r)\varphi(S_s)\mathbf{1}_{\{S_s\leq b(s,Y_s)\}}ds,
		\end{align*}
		where the last equality derives from the application of the It\^{o} formula to the discounted payoff $Z$. In particular, we have
		\begin{align*}
		U_0= M_0=\E[ M_T]&=\E[U_T]-\E\left[\int_0^T e^{-rs}(\mathcal{L}-r)\varphi(S_s)\mathbf{1}_{\{S_s\leq b(s,Y_s)\}}ds\right]\\&=\E[U_T]-\int_0^T e^{-rs}\E[(\delta S_s  -rK)\mathbf{1}_{\{S_s\leq b(s,Y_s)\}}]ds.
		\end{align*}
		
		The assertion follows  recalling that $U_0=P(0,S_0,Y_0)$ and $\E[U_T]=\E[Z_T]=\E[e^{-rT}(K-S_T)_+]$, which corresponds to the price $P_e(0,S_0,Y_0)$ of an European put with maturity $T$ and strike price $K$. 
\end{proof}
		\subsection{Smooth fit}
In this section we analyse the behaviour of the derivatives of the value function with respect to the $s$ and $y$ variables on the boundary of the continuation region.  In other words, we prove a weak formulation of the so called smooth fit principle. 

In order to do this,  we need  two technical lemmas whose proofs can be found in the appendix. The first one is a general result about  the behaviour of the  trajectories of the CIR process.
\begin{lemma}\label{lemmaY}
	For all $y\geq 0$ we have, with probability one,
	$$
	\limsup_{t\downarrow 0 }\frac{Y^y_t-y}{\sqrt{2t\ln\ln(1/t)}}=-	\liminf_{t\downarrow 0 }\frac{Y^y_t-y}{\sqrt{2t\ln\ln(1/t)}}=\sigma\sqrt{y}.
	$$
\end{lemma}

The second one is a result about the behaviour of the trajectories of a standard Brownian motion. 
\begin{lemma}\label{lemmaBM}
	Let $(B_t)_{t\geq 0}$ be a standard Brownian motion and let  $(t_n)_{n\in\N}$ be a deterministic sequence of positive numbers with $\lim_{n\rightarrow\infty}t_n=0$. We have, with probability one, 

\begin{equation}\label{detILL0}
\liminf_{n\rightarrow \infty }\frac{B_{t_n}}{\sqrt{t_n}}=-\infty
\end{equation}
\end{lemma}


We are now in a position to prove the following smooth fit result.
\begin{proposition}\label{sfx}
		For any $(t,y)\in [0,T)\times [0,\infty)$ we have $\frac{\partial}{\partial s} P(t,b(t,y),y)=\varphi'(b(t,y))$.
		\end{proposition}
		\begin{proof}
			The general idea of the proof goes back to \cite{B} for the Brownian motion (see also \cite[Chapter 4]{PS}).  Without loss of generality we can fix $t=0$. 
		Note that, for $h>0$, since $b(0,y)-h\leq b(0,y)$, we have
		$$
	\frac{	P(0,b(0,y)-h,y)-P(0,b(0,y),y)}{h}= \frac{	\varphi(b(0,y)-h)-\varphi (b(0,y))}{h},		
		$$
		so that, since $\varphi$ is continuously differentiable near $b(0,y)$, $		\frac{\partial ^-}{\partial s}P(0,b(0,y),y)=\varphi'(b(0,y))$. 
		
		On the other hand, for $h>0$	small enough, 	since $P\geq \varphi$ and $P(0,b(0,y),y)=\varphi(b(0,y))$, we get
		\begin{align*}
		\frac{	P(0,b(0,y)+h,y)-P(0,b(0,y),y)}{h}\geq \frac{	\varphi(b(0,y)+h)-\varphi (b(0,y))}{h},
		\end{align*}
	so that 
		\begin{align*}
		\liminf_{h\downarrow0}  \frac{	P(0,b(0,y)+h,y)-P(0,b(0,y),y)}{h} \geq \varphi'(b(0,y)).
		\end{align*}
		Now, for the other inequality, we consider the optimal stopping time related to $P(0,b(0,y)+h,y)$, i.e.
		\begin{align*}
		\tau_h=\inf\{ t\in [0,T) \mid S_t^{0,b(0,y)+h,y}< b(t,Y^y_t)\}\wedge T=\inf\left \{ t\in [0,T) \mid M_t^{y}\leq\frac{b(t,Y^y_t)}{b(0,y)+h}\right\}\wedge T,
		\end{align*}
		where $M^y_t=S^{1,y}_t$.
Recall that $P(0,b(0,y),y)\geq \E\big(e^{-r\tau_h}\varphi(b(0,y)M^y_{\tau_h})\big)$, so we can write
		\begin{align*}
		\frac{ 	P(0,b(0,y)+h,y)-P(0,b(0,y),y)}{h}&= \frac{\E \left(     e^{-r\tau_h}  \varphi((b(0,y)+h)M^y_{\tau_h}\right)  -   P(0,b(0,y),y)}{h}\\&\leq \E\left(   e^{-r\tau_h}   \frac{        \varphi\left((b(0,y)+h)M^y_{\tau_h}\right)  -   \varphi\left(b(0,y)M^y_{\tau_h}\right)}{h}     \right).
		\end{align*}
		Assume for the moment  that 
			\begin{equation}\label{ot}
			\lim_{h\rightarrow 0 }\tau_h=0,\quad a.s.
			\end{equation}
so we have
		$$
		\lim_{h\downarrow 0 }   \frac{        \varphi((b(0,y)+h)M^y_{\tau_h})\big)  -   \varphi(b(0,y)M^y_{\tau_h})}{h}=\varphi'(b(0,y)).
		$$
		Moreover, recall that $ M^y_{\tau_h}\leq \frac{b(t,Y^y_t)}{b(0,y)+h}\leq \frac K {b(0,y)}$ if $\tau_h<T$ and $ M^y_{\tau_h}=M^y_T$ if $\tau_h=T$. Therefore, by using the fact that $\varphi$ is Lipschitz continuous and the dominated convergence, we obtain
		$$
		\limsup_{h\downarrow 0 } \frac{ 	P(0,b(0,y)+h,y)-P(0,b(0,y),y)}{h} \leq \varphi'( b(0,y))
		$$
		and the assertion is proved. 
		
	It remains to prove \eqref{ot}. Since $t\mapsto b(t,y) $ is nondecreasing, if $	M^y_t< \frac{b(0,y)}{b(0,y)+h}$ and $Y^y_t= y$,
		we have
		$$
		M^y_t< \frac{b(0,y)}{b(0,y)+h} \leq \frac{b(t,Y^y_t)}{b(0,y)+h},
		$$
		so that
\begin{equation}
		\tau_h \leq 	\inf \left\{  t\geq 0 \mid  M^y_t< \frac{b(0,y)}{b(0,y)+h} \mbox{ \& } Y^y_t= y \right\}.
\end{equation}
We now show that we can find a sequence $t_n \downarrow 0 $ such that  $Y^y_{t_n}=0$ and $ M^y_{t_n}<1$.
	First, recall that with a standard transformation we can write
	\begin{equation}\label{incorrelation}
	\begin{cases}
	\frac{dS_t}{S_t}=(r-\delta)dt +  \sqrt{Y_t}(\sqrt{1-\rho^2}d\bar W_t+\rho dW_t),\qquad&S_0=s>0,\\ dY_t=\kappa(\theta-Y_t)dt+\sigma\sqrt{Y_t}dW_t, &Y_0=y\geq 0,
	\end{cases}
	\end{equation}
	where $\bar W$ is a standard  Brownian motion independent of $W$.	Set $\Lambda_t^y=\ln M^y_t$. We deduce from Lemma \ref{lemmaY} that there exists a sequence  $t_n\downarrow 0$ such that  $Y^y_{t_n} =y$ $\P_y$-a.s. . Therefore,  from \eqref{incorrelation} we can write  $\int_0^{t_n}\sqrt{Y^{y}_s}dW_s=-\frac \kappa \sigma \int_0^{t_{n}}(\theta-Y^{y}_s)ds$ for all $n\in \N$. So, we have
	$$
	\Lambda^y_{t_{n}}=(r-\delta)t_{n}-\int_0^{t_{n}} \frac{Y^y_s}{2}   ds +\sqrt{1-\rho^2}\int_0^{t_{n}}\sqrt{Y^y_s} d\bar W_s -\frac {\rho\kappa} \sigma \int_0^{t_{n}}(\theta-Y^{y}_s)ds .
	$$
Conditioning with respect to $W$  we have 
\begin{align*}
\liminf_{n\rightarrow \infty} \Lambda_{t_n}^y=&\liminf_{n\rightarrow\infty}	\frac{(r-\delta)t_{n}}{\sqrt{\int_0^{t_{n}}Y^{y}_sds}}-\frac{\int_0^{t_{n}} \frac{Y^y_s}{2}   ds}{   \sqrt{\int_0^{t_{n}}Y^{y}_sds}        } +\frac{\sqrt{1-\rho^2}\int_0^{t_{n}}\sqrt{Y^y_s} d\bar W_s  }{\sqrt{\int_0^{t_{n}}Y^{y}_sds} }- \frac{\frac {\rho\kappa} \sigma \int_0^{t_{n}}(\theta-Y^{y}_s)ds }{\sqrt{\int_0^{t_{n}}Y^{y}_sds} }\\
&=\liminf_{n\rightarrow\infty}	\frac{(r-\delta)t_{n}}{\sqrt{\int_0^{t_{n}}Y^{y}_sds}}-\frac{\int_0^{t_{n}} \frac{Y^y_s}{2}   ds}{   \sqrt{\int_0^{t_{n}}Y^{y}_sds}        } +\frac{\sqrt{1-\rho^2}\tilde  W_{ \int_0^{t_{n}}Y^y_sds }}{\sqrt{\int_0^{t_{n}}Y^{y}_sds} }- \frac{\frac {\rho\kappa} \sigma \int_0^{t_{n}}(\theta-Y^{y}_s)ds }{\sqrt{\int_0^{t_{n}}Y^{y}_sds} }=-\infty,
\end{align*}
where  we have used the Dubins-Schwartz Theorem and we have applied Lemma \ref{lemmaBM}  to the standard Brownian motion $\tilde W$ and the sequence  $\sqrt{\int_0^{t_{n}}Y^{y}_sds}$ which can be considered deterministic.

We deduce that, up to extract a subsequence of $t_n$, we have $\Lambda^y_{t_n}<0$ and,  as a consequence, $M^y_{t_n}<1$. Therefore, for any any fixed $n$, there exists  $h$ small enough such that 
$M^y_{t_{n}} <  \frac{b(0,y)}{b(0,y)+h}$  so that, by definition,  $\tau_h\leq t^n$. We conclude the proof passing to the limit as $n$ goes to infinity.

		\end{proof}
		As regards the derivative with respect to the $y$ variable, we have the following result.
		\begin{proposition}\label{prop _sfy}
	If $2\kappa\theta\geq \sigma^2$, for any $(t,y)\in [0,T)\times (0,\infty)$ we have  $\frac{\partial}{\partial y} P(t,b(t,y),y)=0$.
		\end{proposition}
		\begin{proof}
		Again we fix $t=0$ with no loss of generality. Since $y\rightarrow P(t,s,y)$ in nondecreasing, for any $h>0$ we have $P(0,b(0,y),y-h)\leq P(0,b(0,y),y)=\varphi(b(0,y))$ so that  $P(0,b(0,y),y-h)=\varphi(b(0,y))$. Therefore,
			\begin{align*}
			\frac{	P(0,b(0,y),y-h)-P(0,b(0,y),y)}{h}= 0,
			\end{align*}
			hence $\frac{\partial^-}{\partial y}P(0,b(0,y),y)=0$.
			On the other hand, since $y\mapsto P(t,x,y)$ is nondecreasing, for any $h>0$ we have
			\begin{align*}
		\liminf_{h\downarrow 0}	\frac{	P(0,b(0,y),y+h)-P(0,b(0,y),y)}{h}\geq 0, 
			\end{align*}
	 To prove the other inequality, we consider the stopping time related to $P(0,b(0,y),y+h)$, that is
		$$
		\tau_h=\inf\left\{ t\in [0,T) \mid S_t^{0,b(0,y),y+h}<b(t,Y^{y+h}_t)\right \}\wedge T=\inf\left \{ t\in [0,T)  \mid M_t^{y+h}<\frac{b(t,Y^{y+h}_t)}{b(0,y)}\right\}\wedge T
		$$
		and we assume for the moment that
	\begin{equation}\label{limtauh}
	\lim_{h\rightarrow0}\tau_h=0.
	\end{equation}
	We have
					\begin{equation}\label{stimadery}
					\begin{split}
					\frac{ 	P(0,b(0,y),y+h)-P(0,b(0,y),y)}{h}&= \frac{\E \left(     e^{-r\tau_h}  \varphi\left(b(0,y)M^{y+h}_{\tau_h}\right)\right)  -   P(0,b(0,y),y)}{h}\\
					&\leq \E\left[  e^{-r\tau_h}   \frac{        \varphi\left(b(0,y)M^{y+h}_{\tau_h}\right)  -   \varphi(b(0,y)M^y_{\tau_h})}{h}     \right]\\
					&\leq K    \frac{    \E  \left[ \left| M^{y+h}_{\tau_h}  -   M^y_{\tau_h}\right|\right]}{h},
					\end{split}
					\end{equation}
					where the last inequality follows from the fact that $\varphi$ is Lipschitz continuous  and $b(0,y)\leq K$.

Now,  if the Feller condition $2\kappa\theta\geq\sigma^2$ is satisfied,   we can write
$$
M^{y+h}_{t}  -   M^y_{t}=\int_y^{y+h} \left(\int_0^t \frac{\dot{Y}^\zeta_s}{2\sqrt{Y^\zeta_s}}dB_s-\frac 1 2 \int_0^t \dot{Y}^\zeta_sds   \right) e^{(r-\delta)t-\int_0^t \frac {Y^\zeta_s}2 ds+\int_0^t \sqrt{Y^\zeta_s}dB_s   } d\zeta.
$$
Fix $\zeta$ and observe that the exponential process $e^{-\int_0^t \frac {Y^\zeta_s}2 ds+\int_0^t \sqrt{Y^\zeta_s}dB_s   } $ satisfies the assumptions od  the Girsanov Theorem, namely it is a martingale. Therefore, we can introduce a new probability measure $\hat \P$ under which the process $\hat W_t=W_t-\int_0^t\sqrt{Y_s}ds$ is a standard Brownian motion. If we denote by $\hat \E$ the expectation under the probability $\hat \P$, substituting in \eqref{stimadery} and using \eqref{stimaperholder} we get
\begin{align*}
&\frac{ 	P(0,b(0,y),y+h)-P(0,b(0,y),y)}{h}
\leq\frac{e^{rT}K}{h}\int_y^{y+h} d\zeta\hat{\E}\left[\left|\int_0^{\tau_h} \frac{\dot{Y}^\zeta_s}{2\sqrt{Y^\zeta_s}}d\hat{W_s}\right|\right]\\ 
&\quad\leq \frac{e^{rT}K}{h}\int_y^{y+h} d\zeta\left(\hat{\E}\left[\int_0^{\tau_h} \left(\frac{\dot{Y^\zeta_s}}{2\sqrt{Y^\zeta_s}}\right)^2ds\right]\right)^{1/2}
\leq \frac {e^{rT}K} h \int_y^{y+h}\frac{1}{2\sqrt \zeta}\hat{\E}[\sqrt{\tau_h}]d\zeta
\end{align*}
which tends to $0$ as $h$ tends to $0$.

Therefore, as in the proof of Proposition \ref{sfx},  it remains to prove that $\lim_{h\downarrow 0}\tau_h= 0$. In order to do this, we can proceed as follows.
Again, set
$$
\Lambda^y_t=\ln (M^y_t)=(r-\delta)t-\frac 1 2 \int_0^t
Y^y_sds+\int_0^t\sqrt{Y^y_s}dW_s,
$$
so that 
$$
\tau_h=\inf\left\{t\in [0,T)\mid \Lambda_t^{y+h}  \leq \ln\left(\frac{ b(t,Y^{y+h}_t)}{b(0,y)}\right) \right \}\wedge T.
$$
We deduce from Lemma \eqref{lemmaY} that, almost surely, there exist two sequences $(t_n)_n$ and $(\hat t_n)_n$ which converge to 0 with $0<t_n<\hat t_n$ and such that
$$
Y^y_{t_n}=y, \qquad \mbox{ and, for }t\in (t_n,\hat t_n),\quad Y_t<y. 
$$
In fact, it is enough to consider a sequence $(\hat t_n)_n$  such that $\lim_{n\rightarrow\infty}\hat t_n=0$ and $Y_{\hat t_n}<y$ and define $t_n=\sup\{t\in [0,\hat t_n) \mid Y^y_t=y   \}$.

Proceeding as in the proof of Proposition \ref{sfx}, up to extract a subsequence we can assume 
$$
\Lambda^y_{t_n}<0.
$$

On the other hand, up to extracting a subsequence of $h$ converging to $0$, we can assume that, almost surely,
\[
\lim_{h\downarrow 0}\sup_{t\in[0,T]}\left|Y^{y+h}_t-Y^y_t\right|=\lim_{h\downarrow 0}\sup_{t\in[0,T]}\left|\Lambda^{y+h}_t-\Lambda^y_t\right|=0.
\]

Now, let us fix $n\in\N$. For $h$ small enough,  there exists $\delta>0$ such that
\[
\Lambda^{y+h}_{t}<0, \qquad t\in (t_n-\delta, t_n+\delta).
\]
Then, for any $\tilde{t}_n\in (t_n-\delta, t_n+\delta)\cap (t_n,\hat t_n)$, we have at the same time
 $\Lambda^{y+h}_{\tilde{t}_n}<0$ and, since $Y^{y}_{\tilde{t}_n}<y$, $Y^{y+h}_{\tilde{t}_n}<y$ for $h$ small enough.
Recalling that $t\mapsto b(t,y)$ is nondecreasing and $y\mapsto b(t,y)$ is nonincreasing, we deduce that
\[
b(\tilde{t}_n,Y^{y+h}_{\tilde{t}_n})\geq b(0,Y^{y+h}_{\tilde{t}_n})\geq b(0,y).
\]

Therefore
\[
\Lambda^{y+h}_{\tilde{t}_n}\leq \ln\left(\frac{b(\tilde{t}_n,Y^{y+h}_{\tilde{t}_n})}{b(0,y)}\right)
\]
and, as a consequence, $\tau_h\leq \tilde{t}_n\leq \hat t_n$ so \eqref{limtauh} follows.

		\end{proof}

\section{Appendix: some proofs}

We devote the appendix to the proof of some technical results used in this paper.
\subsection{Proofs of Section \ref{sect-monotony}}
\begin{proof}[Proof of Lemma \ref{lemmasup_diff}]
As in \cite[Section IV.3]{IW}, we introduce a sequence $1 > a_1 > a_2 >\dots >a_m >\dots >0$  defined by 
		$$ 	
		\int^1_{a_1} \frac 1 u du =1, \dots, \int_{a_m}^{a_{m-1}} \frac 1 u du = m, \ \dots .
		$$
		We have that $a_m$ tends to $0 $ as $m$ tends to infinity. Let $(\eta_m)_{m\geq 1}$, be a family of continuous functions such that $$\supp \eta_m\subseteq (a_m, a_{m-1}), 
		\quad
		0 \leq \eta_m(u) \leq \frac{2}{um}, \quad  \int_{a_m}^{a_{m-1}}\eta_m(u)du=1.
		$$
		Moreover, we set
		$$
		\phi_m(x) := \int_0^{|x|}dy \int_0^y \eta_m(u)du, \qquad x \in \R .
		$$
		It is easy to see  that $\phi_m \in C^2(\R)$, $|\phi_m^{'} | \leq 1$ and $\phi_m(x)\uparrow |x| $ as $m \rightarrow \infty$.
	Fix $t\in[0,T]$.	Applying It\^{o}'s formula and passing to the expectation we have, for any $m \in \N$,
		\begin{equation}\label{II}
			\E[\phi_m(Y^n_t-Y_t)]= \kappa\int_{0}^{t} \! \E\left[\phi_m^{'}(Y^n_s-Y_s)(Y_s- f^2_n(Y^n_s))\right]ds + \frac {\sigma^2} 2 \int_{0}^{t}  \!  \E \left[       \phi_m^{''}(Y^n_s-Y_s)(f_n(Y^n_s)- \sqrt{Y_s})^2  \right]ds 
			\end{equation}
		Let us analyse the  right hand term in \eqref{II}. Since $|\phi_m^{'}|\leq 1$, we have
		\begin{align*}
			&  \left| \kappa \int_{0}^{t} \E\left[\phi_m^{'}(Y^n_s-Y_s)(Y_s- f_n^2(Y^n_s))\right]ds \right|
		\leq \kappa \int_{0}^{t} \E\left[ |f_n^2(Y^n_s)-Y^n_s| \right]ds + \kappa \int_{0}^{t} \E\left[ | Y^n_s- Y_s|\right]ds
		\end{align*}
		On the other hand,
		\begin{align*}
			&\left|\frac {\sigma^2} 2 \int_{0}^{t}    \E \left[       \phi_m^{''}(Y^n_s-Y_s)(f_n(Y^n_s)- \sqrt{Y_s})^2\right]ds\right|   \\
			& \quad \leq  \sigma^2\int_{0}^{t}    \E \left[     |  \phi_m^{''}(Y^n_s-Y_s)|(f_n(Y^n_s)- f(Y^n_s)^2]ds  \right] + \sigma^2 \int_{0}^{t}    \E \left[  |     \phi_m^{''}(f(Y^n_s)-Y_s)|(\sqrt{Y^n_s}- \sqrt{Y_s})^2\right]ds \\
			&\quad \leq \sigma^2\int_{0}^{t}    \E \left[    \frac{2}{m|Y^n_s-Y_s|}(f_n(Y^n_s)- f(Y^n_s))^2    \mathbf{1}_{\{a_m\leq |Y^n_s-Y_s!\leq a_{m-1}\}} ]ds  \right] + \sigma^2\int_0^t \E \left[  \frac{2}{m|Y^n_s-Y_s|}  |Y^n_s-Y_s|\right]ds\\
			&\quad\leq\frac{2\sigma^2}{ma_m}  \int_{0}^{t}    \E \left[    (f_n(Y^n_s)- f(Y^n_s))^2  ]ds  \right]  + \frac{2\sigma^2t}{m}.
		\end{align*}
		Observe that, if $|x| \geq a_{m-1}$,
		$$
		\phi_m(x)\geq \int_{a_{m-1}}^{|x|}dy= |x|-a_{m-1}.
		$$
		Therefore,  for any $m$ large enough,  
		\begin{align*}
			\E[\left| Y^n_t-Y_t\right|] &\leq  \kappa   \int_0^t \E[| Y^n_s-Y_s |]ds+ \kappa \int_{0}^{t} \E\left[ |f_n^2(Y^n_s)-f^2(Y^n_s)| \right]ds+\frac{2\sigma^2}{ma_m}  \int_{0}^{t}    \E \left[    (f_n(Y^n_s)- f(Y^n_s))^2  ]ds  \right]  \\&\qquad+  \frac{2\sigma^2t}{m}+ a_{m-1}.
		\end{align*}
	Recall that $f_n(y)\rightarrow f(y)= \sqrt{y^+}$ locally uniformly and that $Y^n$ has continuous paths. Moreover, since $f_n^2$ is ì Lipschitz continuous uniformly in $n$, we have that $f^2_n(x)\leq A(|x|+1)$ with $A$ independent of $n$. Therefore, it is easily to see that for any $p>1$ there exists $C>0$ independent of $n$ such that
		\begin{equation}\label{moments}
			\E\left[\sup_{t\in [0,T]} |Y^n_t|^p  \right]\leq C.
			\end{equation}
			Fix now $m\in\N$. By using Lebesgue's Theorem,
		we deduce that there exist 
			$\bar n$ and $C>0$ such that, for every $ n\geq \bar n$,
		$$
		\E[\left| Y^n_t-Y_t\right|] <\kappa \int_{0}^{t} \E[\left| Y^n_s-Y_s\right|]  + \kappa Ca_m + \frac{2\sigma^2}{ma_m} 	a_m +  \frac{2\sigma^2t}{m}+ a_{m-1}.. $$ 
		We can now apply  Gronwall's inequality. Passing to the limit as $m\rightarrow\infty$ and  recalling that  $\lim_{m\rightarrow \infty }a_m=0$, we can deduce  that 
	 \begin{equation}\label{convfixedt}
	\lim_{n\rightarrow \infty} \E[\left| Y^n_t-Y_t\right|] = 0.
	 	\end{equation} 
	
		Now, note that 
	\begin{equation}\label{relsup}
	\sup_{t \in [0,T]}|Y^n_t-Y_t|\leq \kappa \int_0^T|Y_s-Y^n_s|ds  + \sup_{t \in [0,T]}\left|    \int_0^t  (\sqrt{Y_s}-f_n(Y^n_s))dW_s\right|
	\end{equation}
The first term in the right hand side of \eqref{relsup} converges to 0 in probability  thanks to \eqref{convfixedt}, so it is enough to prove that the second term converges to 0. We have
\begin{equation}\label{stimasup}
\E\left[    \sup_{t \in [0,T]}\left|    \int_0^t  (\sqrt{Y_s}-f_n(Y^n_s))dW_s\right|  \right]\leq\left(\int_0^T \E[ |\sqrt{Y_s}-f_n(Y^n_s)|^2]ds\right)^{\frac 1 2 }
\end{equation}
and 
\begin{align*}
\E\left[ |\sqrt{Y_s}-f_n(Y^n_s)|^2	\right]&\leq 2\E\left[ |\sqrt{Y_s}-\sqrt{Y^n_s}|^2\right]+2\E\left[ |\sqrt{Y^n_s}-f_n(Y^n_s)|^2\right]\\&\leq 2\E\left[ |Y_s-Y^n_s|\right]+2\E\left[ |\sqrt{Y^n_s}-f_n(Y^n_s)|^2\right].
\end{align*}
Therefore, we can conclude that \eqref{stimasup} tends to 0 as $n$ goes to infinity by using \eqref{convfixedt} and the Lebesgue Theorem so that \eqref{supY} is proved.

		As regards \eqref{1}, for every $n\in\N$ we have
		$$
		X^n_t= x +\int_0^t \left( r-\delta-\frac {f_n^2(Y^n_s)} 2 \right)ds + \int_0^t  f_n(Y^n_s)dB_s,
		$$
		so that
		\begin{equation}\label{calcoloperX}
	\sup_{t \in [0,T]}|X^n_t-X_t|  \leq \frac 12 \int_0^T |f^2_n(Y^n_s)- Y_s| ds + 	\sup_{t \in [0,T]}  \left|\int_0^t (f_n(Y^n_s)- \sqrt{Y_s} )dB_s \right|    .
		\end{equation}
		It is enough to show that the two terms in the right hand side  of \eqref{calcoloperX} converge to 0 in probability.
		
		Concerning the first term, note that,
		since $Y$ has continuous paths, for every $ \omega \in \Omega, \, Y_{[0,T]}(\omega)$ is a compact set and 
		$
		K:=\{    x | d(x,Y_{[0,T]}) \leq 1     \}
		$
		is compact as well.
		For $n$ large enough,  $Y^n$ lies in $K$, so 
		\begin{align*}
			& \int_0^T |f^2_n(Y^n_s)- f^2(Y_s)| ds  \leq  \int_0^T |f^2_n(Y^n_s)- f^2(Y^n_s)| ds + \int_0^T |f^2(Y^n_s)- f^2(Y_s)|ds,
		\end{align*}
		which goes to $0$	as $n$ tends to infinity, since $f^2_n\rightarrow f^2$ locally uniformly and $f^2$ is a continuous function. 
		
		On the other hand, for the second term in the right hand side of \eqref{calcoloperX}, we have
		$$
		\E\left[	\sup_{t \in [0,T]}  \left|\int_0^t f(Y^n_s)- \sqrt{Y_s} dW_s \right|\right]\leq \left(  \int_0^T \E[(f(Y^n_s)- \sqrt{Y_s}  )^2] ds \right)^{\frac 1 2}
		$$
and we can prove with the usual arguments that the last term goes to 0.

\end{proof}
\subsection{Proofs of Section \ref{sect-put}}
\begin{proof}[Proofs of Lemma \ref{support}]
To simplify the notation we pass to the logarithm and we prove the assertion for the pair $(X,Y)$. We can get rid of the correlation between the  Brownian motions  with a standard transformation, getting
\begin{equation*}
\begin{cases}
dX_t=(r-\delta-\frac 1 2 Y_t)dt +  \sqrt{Y_t}(\sqrt{1-\rho^2}d\bar W_t+\rho dW_t),\qquad&X_0\in\R,\\ dY_t=\kappa(\theta-Y_t)dt+\sigma\sqrt{Y_t}dW_t, &Y_0\geq 0,
\end{cases}
\end{equation*}
where $\bar W$ is a standard Brownian motion independent of $W$. Moreover, from the SDE satisfied by $Y$ we deduce $\int_0^t \sqrt{Y_s}dW_s=\frac 1 \sigma\left(Y_t-Y_0-\int_0^t\kappa(\theta-Y_s)ds\right)$. Conditioning with respect to $Y$, it suffices to prove that, for every continuous function $m: [0,T]\rightarrow \R$ such that $m(0)=X_0$ and for every $\epsilon>0$ we have
\begin{equation}
\label{supX}
\P\left(\sup_{t\in [0,T]}|X_t-m(t)|<\epsilon\mid Y \right)>0,
\end{equation}
and
\begin{equation}
\label{supY}
\P\left(\sup_{t\in [0,T]}|Y_t-Y_0|<\epsilon \right)>0.
\end{equation}
As regards  \eqref{supX}, by using the Dubins-Schwartz Theorem, there exists a Brownian motion $\tilde W$ such that 
\begin{align*}
&\P\left(\sup_{t\in [0,T]}\left |x+\int_0^t \left(r-\delta-\frac{Y_s} 2-\frac {\rho\kappa} \sigma (\theta-Y_s) \right) ds    +\frac \rho \sigma(Y_t-y) +\sqrt{1-\rho^2}  \int_0^t\sqrt{Y_s}d\bar W_s  -m(t)\right|<\epsilon\mid Y \right)\\
&=\P\left(\sup_{t\in [0,T]}\left|\sqrt{1-\rho^2}  \int_0^t\sqrt{Y_s}d\bar W_s  -\tilde m(t)\right|<\epsilon\mid Y \right)=\P\left(\sup_{t\in [0,T]}\left|\sqrt{1-\rho^2}  \tilde W_{\int_0^tY_sds}  -\tilde m(t)\right|<\epsilon\mid Y \right),
\end{align*}
where  $\tilde m(t)=m(t)-x-\int_0^t \left(r-\delta-\frac{Y_s} 2-\frac {\rho\kappa} \sigma (\theta-Y_s) \right) ds    -\frac \rho \sigma(Y_t-y) $ is a continuous function which, conditioning w.r.t. $Y$, can be considered deterministic. Then, \eqref{supX} follows by the support theorem for Brownian motions.

In order to prove \eqref{supY}, we distinguish two cases.  Assume  first that $Y_0=y_0>0$ and, for $a\geq 0$, define the stopping time
 $$
 T_{a}=\inf \left\{ t>0\mid Y_t=a 	\right \}.
 $$ 
 Moreover, let us consider the function 
 $$
 \eta(y)=\begin{cases}
\sqrt{y},\qquad&\mbox{ if } y> \frac{y_0}{2},\\
\frac{\sqrt{y_0}} 2 \qquad&\mbox{ if } y\leq \frac{y_0}{2},
 \end{cases}
 $$
 and the process $(\tilde Y_t)_{t\in[0,T]}$, solution to the uniformly elliptic  SDE
 $$
 d\tilde Y_t=\kappa(\theta-\tilde Y_t)dt+\sigma\eta(\tilde Y_t)dW_t,\qquad \tilde Y_0=Y_0.
 $$
 It is clear that $Y_t=\tilde Y_t$ on the set $\left\{t\leq T_{\frac {y_0} 2}\right\}$  so  we have, if $\epsilon<\frac{y_0}2$,
\begin{align*}
 \P\left(\sup_{t\in [0,T]}|Y_t-Y_0|<\epsilon \right)= \P\left(\sup_{t\in [0,T]}|\tilde Y_t-Y_0|<\epsilon \right),
\end{align*}
 where the last inequality follows from the classical Support Theorem for  uniformly elliptic diffusions (see, for example,  \cite{SV1}).
 
 On the other hand, if we assume $Y_0=0$, then we can write
 $$
 \P\left(\sup_{t\in [0,T]} Y_t<\epsilon \right)=\P\left(  T_{\frac \epsilon 2}  \geq T \right)+\P\left(  T_{\frac \epsilon 2}  < T , \forall t\in \left[T_{\frac \epsilon 2}  ,T\right] Y_t<\epsilon \right).
 $$
 Now, if $\P\left( T_{\frac \epsilon 2} < T\right)>0$, we can deduce that the second term in the right hand side is positive using the strong Markov property and the same argument we have used before  in the case with $Y_0\neq 0$. Otherwise, $\P\left(  T_{\frac \epsilon 2}  \geq T \right)=1$ which concludes the proof.
\end{proof}

%
%
\begin{proof}[Proof of Lemma \ref{lemmaY}]	
We have
	\begin{align*}
	Y^y_t-y&=\kappa\int_0^t(\theta-Y^y_s)ds+\sigma\int_0^t\sqrt{Y^y_s}dW_s\\
	&=\sigma\sqrt{y}W_t+\kappa\int_0^t(\theta-Y^y_s)ds+\sigma\int_0^t\left(\sqrt{Y^y_s}-\sqrt{y}\right)dW_s,
	\end{align*}
so it is enough to prove that, if $(H_t)_{t\geq 0}$ is a predictable process such that 
	$\lim_{t\downarrow 0} H_t=0$ a.s., we have
	\[
	\lim_{t\downarrow 0}\frac{\int_0^tH_s dWs}{\sqrt{2t\ln\ln(1/t)}}=0 \mbox{ p.s.}
	\]
This follows by using standard arguments, we include a proof for the sake of completeness. By using   Dubins-Schwartz inequality we deduce that, if  $f(t)=\sqrt{2t\ln\ln(1/t)}$, for $t$ near to $0$ we have
	\[
	\left|\int_0^tH_s dWs\right|\leq Cf\left(\int_0^tH_s^2ds\right).
	\]
	Let us consider $\varepsilon>0$. For $t$ small enough, we have $\int_0^tH_s^2ds\leq \varepsilon t$ and, since $f$ increases near $0$,
	\[
	\left|\int_0^tH_s dWs\right|\leq Cf\left(\varepsilon t\right).
	\]
We have
	\begin{align*}
	\frac{f^2(\varepsilon t)}{f^2(t)}&=\frac{\varepsilon t\ln\ln(1/\varepsilon t)}{t\ln\ln(1/t)}=\varepsilon\frac{\ln\left(\ln(1/t)+\ln(1/\varepsilon)\right)}{\ln\ln(1/t)}\leq \varepsilon\frac{\ln\left(\ln(1/t)\right)+\frac{\ln(1/\varepsilon)}{\ln(1/t)}}{\ln\ln(1/t)}=\varepsilon\left(1+\frac{\ln(1/\varepsilon)}{\ln(1/t)\ln\ln(1/t)}\right),
	\end{align*}
	where we have used the inequality $\ln(x+h)\leq \ln(x)+\frac{h}{x}$ (for $x,h>0$).
	Therefore $\limsup_{t\downarrow 0}\frac{f(\varepsilon t)}{f(t)}\leq \sqrt{\varepsilon}$ and the assertion follows.
\end{proof}
	\begin{proof}[Proof of Lemma \ref{lemmaBM}]
With  standard inversion arguments, it suffices to prove  that,  for a sequence $t_n$ such that $\lim_{n\rightarrow\infty }t_n=\infty$, we have, with probability one,
\begin{equation}\label{detILLinf}
	\limsup_{n\rightarrow \infty }\frac{B_{t_n}}{\sqrt{t_n}}=+\infty.
\end{equation}
		The assertion is equivalent to
		$$
		\P\left(\limsup_{n\rightarrow\infty}\frac{B_{t_n}}{\sqrt{t_n}}\leq c\right)=0, \qquad c>0,
		$$
		that is
		$$
		\P\left(\bigcup_{m\geq1}\bigcap_{n\geq m} \left\{\frac{B_{t_n}}{\sqrt{t_n} }\leq c \right\}\right)=0, \qquad c>0.
		$$
		Therefore, it is sufficient to prove that $	\P\left(\bigcap_{n\geq m} \left\{\frac{B_{t_n}}{\sqrt{t_n} }\leq c \right\}\right)=0$ for every $m\in \N$ and $c>0$. Take, for example, $m=1$ and consider the random variables $\frac{B_{t_1}}{\sqrt{t_1}}$ and $\frac{B_{t_n}}{\sqrt{t_n}}$, for some $n>1$. Then,
		$$
		\frac{B_{t_1}}{\sqrt{t_1}}, \,\frac{B_{t_n}}{\sqrt{t_n}}\sim \mathcal{N}(0,1),
		$$
		where $\mathcal{N}(0,1)$ is the standard Gaussian law and 
		$$
		\mbox{Cov}\left( \frac{B_{t_1}}{\sqrt{t_1}}, \frac{B_{t_n}}{\sqrt{t_n}}\right)=\frac{t_1\wedge t_n}{\sqrt{t_1t_n}}<
		\sqrt{	\frac{t_1}{t_n}},$$	which tends to $0$ as $n$ tends to infinity.
		We deduce that 
		$$
		\P\left(	\frac{B_{t_1}}{\sqrt{t_1}}\leq c, \frac{B_{t_n}}{\sqrt{t_n}}\leq c\right)\rightarrow \P(Z_1\leq c,Z_2\leq c)=\P(Z_1\leq c)^2,
		$$
		where $Z_1$ and $Z_2$ are independent with $Z_1,\,Z_2\sim \mathcal{N}(0,1)$.
		
		Take  now  $m_n\in \N$ such that $t_{m_n}>nt_n$. Then, we have 
		$$
		\frac{B_{t_1}}{\sqrt{t_1}}, \,\frac{B_{t_n}}{\sqrt{t_n}},\frac{B_{t_{m_n}}}{\sqrt{t_{m_n}}}\sim \mathcal{N}(0,1)
		$$
		and
		$$ 
		\mbox{Cov}\left( \frac{B_{t_1}}{\sqrt{t_1}}, \frac{B_{t_{m_n}}}{\sqrt{t_{m_n}}}\right),\,
		\mbox{Cov}\left( \frac{B_{t_n}}{\sqrt{t_n}}, \frac{B_{t_{m_n}}}{\sqrt{t_{m_n}}}\right)\leq
		\sqrt{	\frac{t_n}{t_{m_n}}}.
		$$
		which again tends to $0$ ad $n$ tends to infinity. Therefore, we have 
		$$
		\P\left(	\frac{B_{t_1}}{\sqrt{t_1}}\leq c, \frac{B_{t_n}}{\sqrt{t_n}}\leq c,	\frac{B_{t_{m_n}}}{\sqrt{t_{m_n}}}\leq c\right)\rightarrow \P(Z_1\leq c)^3
		$$
		with $Z_1\sim \mathcal{N}(0,1)$. 
		Iterating this procedure, we can find a subsequence $(t_{n_k})_{k\in \N}$ such that $t_{n_k}\rightarrow \infty $ and 
		$$
		\P\left(\bigcap_{k\geq 1} \left\{\frac{B_{t_{n_k}}}{\sqrt{t_{n_k}} }\leq c \right\}\right)=0
		$$
	which proves that $\limsup_{n\rightarrow \infty }\frac{B_{t_n}}{\sqrt{t_n}}=+\infty$. 
		
	\end{proof}

\end{document}